\documentclass[11pt]{article}
\usepackage{mathpazo}
\usepackage{times}

\usepackage{amsthm}
\usepackage{amsmath}
\usepackage{amsfonts}
\usepackage{amssymb}
\usepackage{mathrsfs} 
\usepackage{latexsym}
\usepackage{enumerate}

\usepackage{cite}      
 
\usepackage{upref}

\usepackage{graphicx}

\usepackage{comment}  
\usepackage{fullpage}

\usepackage{tikz}
\usetikzlibrary{arrows,calc}

\title{\Huge\bf $\,$\\[-6.00ex]
\hspace*{-1.20ex}\mbox{Domination Mappings into the Hamming~Ball:}\\ 
Existence, Constructions, and Algorithms\\[3.00ex]}

\author{
{\Large Yeow Meng Chee}\\
   \small School of Physical and Mathematical Sciences\vspace*{-0.50ex}\\
   \small Nanyang Technological University\vspace*{-0.50ex}\\
   \small 21 Nanyang Link, Singapore 637371\vspace*{-0.25ex}\\
   \ttfamily\bfseries\small ymchee@ntu.edu.sg\\[6.0ex]
\and
{\Large Tuvi Etzion}\\
   \small Department of Computer Science\vspace*{-0.50ex}\\
   \small Technion --- Israel Institute of Technology\vspace*{-0.50ex}\\
   \small Technion City, Haifa 32000, Israel\vspace*{-0.25ex}\\
   \ttfamily\bfseries\small etzion@cs.technion.ac.il\\[4.5ex]
\and
{\Large Han Mao Kiah}\\
   \small School of Physical and Mathematical Sciences\vspace*{-0.50ex}\\
   \small Nanyang Technological University\vspace*{-0.50ex}\\
   \small 21 Nanyang Link, Singapore 637371\vspace*{-0.25ex}\\
   \ttfamily\bfseries\small hmkiah@ntu.edu.sg\\[4.5ex]
\and
{\Large Alexander Vardy}\\
   \small Department of Computer Science and Engineering\vspace*{-0.50ex}\\
   \small University of California San Diego\vspace*{-0.50ex}\\
   \small 9500 Gilman Drive, La Jolla, CA\,92093\vspace*{-0.25ex}\\
   \ttfamily\bfseries\small avardy@ucsd.edu\\[6.5ex]
}

\date{\Large July 12, 2018\vspace{6.00ex}}

\newtheorem{theorem}{Theorem}
\newtheorem{lemma}[theorem]{Lemma}
\newtheorem{corollary}[theorem]{Corollary}
\newtheorem{conjecture}{Conjecture}
\newtheorem{proposition}[theorem]{Proposition}
\newtheorem{definition}{Definition}

\theoremstyle{definition}
\newtheorem{example}{Example}

\newcommand{\Tref}[1]{The\-o\-rem\,\ref{#1}}

\newcommand{\Lref}[1]{Lem\-ma\,\ref{#1}}
\newcommand{\Cref}[1]{Co\-ro\-lla\-ry\,\ref{#1}}
\newcommand{\Dref}[1]{Def\-i\-ni\-tion\,\ref{#1}}
\newcommand{\Eref}[1]{Ex\-am\-ple\,\ref{#1}}

\newcommand{\be}[1]{\begin{equation}\label{#1}}
\newcommand{\ee}{\end{equation}}
\newcommand{\eq}[1]{(\ref{#1})}

\renewcommand{\le}{\leqslant}
\renewcommand{\leq}{\leqslant}
\renewcommand{\ge}{\geqslant}
\renewcommand{\geq}{\geqslant}

\newcommand{\F}{\mathbb{F}}
\newcommand{\Z}{\mathbb{Z}}
\newcommand{\bbS}{\mathbb{S}}
\newcommand{\bbB}{\mathbb{B}}
\newcommand{\bbG}{\mathbb{G}}

\newcommand{\bfit}{\bfseries\itshape}

\newcommand{\deff}{\stackrel{\rm def}{=}}
\newcommand{\trunc}[1]{\left\lfloor {#1} \right\rfloor}
\newcommand{\ceil}[1]{\left\lceil {#1} \right\rceil}

\newcommand{\bits}{\{0,1\}}

\DeclareMathOperator{\wt}{wt}
\DeclareMathOperator{\supp}{supp}

\newcommand{\zero}{{\mathbf 0}}

\newcommand{\wts}{(m_1,m_2)\mbox{-}\wt}
\newcommand{\wtr}{(n_1,n_2;\delta)\mbox{-}\wt}

\newcommand{\cB}{{\cal B}}

\newcommand{\cG}{{\cal G}}

\newcommand{\cO}{{\cal O}}

\newcommand{\cR}{{\cal R}}
\newcommand{\cS}{{\cal S}}

\newcommand{\cX}{{\cal X}}

\DeclareMathAlphabet{\mathbfsl}{OT1}{ppl}{b}{it} 

\newcommand{\vzero}{\mathbf{0}} 
\newcommand{\vone}{\mathbf{1}} 

\newcommand{\bA}{\mathbfsl{A}}

\newcommand{\bX}{\mathbfsl{X}}

\newcommand{\aaa}{\mathbfsl{a}} 
\newcommand{\bbb}{\mathbfsl{b}} 
\newcommand{\ccc}{\mathbfsl{c}} 
\newcommand{\eee}{\mathbfsl{e}} 
 
\newcommand{\uuu}{\mathbfsl{u}} 
\newcommand{\vvv}{\mathbfsl{v}}
 
\newcommand{\xxx}{\mathbfsl{x}}
\newcommand{\yyy}{\mathbfsl{y}} 
\newcommand{\zzz}{\mathbfsl{z}}

\usepackage[T1]{fontenc}
\DeclareSymbolFont{Greek}{OML}{ptm}{m}{n}
\DeclareMathSymbol{\varphi}{\mathalpha}{Greek}{'047}
\DeclareMathSymbol{\gamma}{\mathalpha}{Greek}{'015}
\DeclareMathSymbol{\nu}{\mathalpha}{Greek}{'027}

\makeatletter

\gdef\@punct{.\ \ }  
\def\@sect#1#2#3#4#5#6[#7]#8{%
  \ifnum #2>\c@secnumdepth
     \def\@svsec{}
  \else
     \refstepcounter{#1}\edef\@svsec{%
     \ifnum #2>0{{\csname the#1\endcsname}}.\fi%
    \hskip .5em}
  \fi
  \@tempskipa #5\relax
  \ifdim \@tempskipa>\z@
     \begingroup #6\relax
       \@hangfrom{\hskip #3\relax\@svsec}{\interlinepenalty \@M #8\par}
     \endgroup
     \csname #1mark\endcsname{#7}
     \addcontentsline{toc}{#1}{\ifnum #2>\c@secnumdepth\else
          \protect\numberline{\csname the#1\endcsname}\fi#7}
  \else
     \def\@svsechd{#6\hskip #3\@svsec #8\@punct\csname #1mark\endcsname{#7}
     \addcontentsline{toc}{#1}{\ifnum #2>\c@secnumdepth \else
          \protect\numberline{\csname the#1\endcsname}\fi#7}}
  \fi
  \@xsect{#5}}

\def\@ssect#1#2#3#4#5{\@tempskipa #3\relax
  \ifdim \@tempskipa>\z@
    \begingroup #4\@hangfrom{\hskip #1}{\interlinepenalty \@M #5\par}\endgroup
  \else \def\@svsechd{#4\hskip #1\relax #5\@punct}\fi
  \@xsect{#3}}

\makeatother

\begin{document}

\maketitle

\thispagestyle{empty}

\begin{abstract}
\vspace{-.36ex}
\noindent\looseness=-1
The Hamming ball of radius $w$ in $\bits^n$ is the set $\cB(n,w)$ of 
all binary words of length $n$ and Hamming weight at most $w$.
We consider injective mappings $\varphi\!: \bits^m \to \cB(n,w)$ 
with the following~\emph{domination property:} every position $j \in [n]$
is dominated by some position $i \in [m]$, in the sense that 
``switching off'' position $i$ in\, $\xxx \:{\in}\, \bits^m$ necessarily
switches off position $j$ in its image $\varphi(\xxx)$.
This property may be described more precisely in terms of 
a bipartite \emph{domination graph} $G = \bigl([m] \cup [n], E\bigr)$
with no isolated\linebreak vertices; for all $(i,j) \in E$ and all
$\xxx~{\in}\, \bits^m$, we require that $x_i = 0$ implies
$y_j = 0$, where $\yyy = \varphi(\xxx)$.
Al\-though such domination mappings recently found applications
in the context of coding for high-performance interconnects,
to the best of our knowledge, they were not previously studied.

In this paper, we begin with simple necessary conditions for the existence
of an \emph{$(m,n,w)$-domination mapping $\varphi\!: \bits^m \to \cB(n,w)$}.
We then provide several explicit constructions of such mappings,~which 
show that the necessary conditions are also sufficient when $w=1$, 
when $w=2$ and $m$ is odd, or when $m \le 3w$.
One of our main results herein is a proof that the trivial necessary
condition $|\cB(n,w)| \ge 2^m$ 
is, in fact, sufficient for the existence of an $(m,n,w)$-domination 
mapping whenever $m$ is sufficiently large. 
We also present a polynomial-time algorithm that, given any $m$, $n$, 
and $w$, determines whether an $(m,n,w)$-domination mapping exists 
for a domination graph with an equitable degree distribution.
\end{abstract}

\vspace*{6.00ex}
%

\newpage
\setcounter{page}{1}
\renewcommand{\baselinestretch}{1.05}\normalsize

\section{Introduction}
\label{sec:intro}
\label{sec:introduction}
\vspace{-0.56ex}

Given a binary word $\yyy = (y_1,y_2,\ldots,y_n)$, 
the \emph{Hamming weight of $\yyy$} is the number of nonzero 
positions in~$\yyy$. 
Explicitly, \smash{$\wt(\yyy) \deff \bigl\{ j \in [n] \,:\, y_j \ne 0 \bigr\}$}.~
The \emph{Hamming ball of radius $w$ in $\bits^n$} is the set $\cB(n,w)$
of all words of weight at most $w$. Explicitly,
\smash{$\cB(n,w) \deff \bigl\{ \yyy \in \bits^n \,:\, \wt(\yyy) \le w \bigr\}$}.
Given $m \le n$,~we are interested in injective mappings $\varphi$ 
from $\bits^m$ into $\cB(n,w)$ that establish a certain domination 
relationship between positions in $\xxx \,{\in}\, \bits^m$ and positions
in its image $\yyy = \varphi(x)$.
Specifically, one should be able to ``switch off'' every position
$j \,{\in}\, [n]$ in $\yyy$ (that is, ensure that $y_j = 0$) by switching
off a corresponding position $i \,{\in}\, [m]$ in $\xxx$ (that is, setting
$x_i = 0$). More precisely, let $G = \bigl([m] \cup [n], E\bigr)$ be a 
bipartite graph with $m$ left vertices and $n$~right vertices.
If $G$ has no isolated right vertices, we refer to $G$ as
a \emph{\bfit domination graph}.

\begin{definition}
\label{domination-def}
Given an injective map $\varphi\!: \bits^m \to \cB(n,w)$ and 
a graph $G = \bigl([m] \cup [n], E\bigr)$, we say~that~$\varphi$ 
is a {\bfit $G$-domination mapping}, or {\bfit $G$-dominating} in brief, if
\begin{align*}
\forall\, (x_1,x_2,\ldots,x_m) \in \bits^m,
\hspace{0.90ex}
\forall\, (i,j) \in E :
\hspace{36.00ex}\\
\text{if\/ $\varphi(x_1,x_2,\ldots,x_m) = (y_1,y_2,\ldots,y_n)$ and $x_i = 0$,\,
then $y_j = 0$}
\end{align*}
We say that $\varphi$ is an {\bfit $(m,n,w)$-domination mapping} if
there exists a domination 
graph $G = \bigl([m] \cup [n], E\bigr)$, with no isolated right vertices,
such that $\varphi$ is {$G$-dominating}.
\end{definition}

\noindent
{\bf Example\,1.}
\label{Example1}
Let $m = 3$, $n = 4$, $w = 2$, and consider the injective map
$\varphi\!: \bits^3 \to \cB(4,2)$ given by
\be{example-mapping}
\begin{array}{lcr}
\varphi(000)=0000\\
\varphi(001)=0001
\end{array}
\hspace{1.80ex}
\begin{array}{lcr}
\varphi(010)=0010\\
\varphi(011)=0011
\end{array}
\hspace{1.80ex}
\begin{array}{lcr}
\varphi(100)=0100\\
\varphi(101)=0101
\end{array}
\hspace{1.80ex}
\begin{array}{lcr}
\varphi(110)=1000\\
\varphi(111)=1001
\end{array}
\ee
It is easy to verify by inspection that $\varphi$ is $G$-dominating, where 
$G = \bigl([3] \cup [4], E)$
can be chosen as any one of the~following three domination 
graphs:\vspace{-0.54ex}
$$
\scalebox{1.20}{%
\begin{tikzpicture}
\pgfsetlinewidth{0.72pt}
    \node [anchor=east] (L1) at (0.06, 1.00) {{\scriptsize$1$}};
    \node [anchor=east] (L2) at (0.06, 0.50) {{\scriptsize$2$}};
    \node [anchor=east] (L3) at (0.06, 0.00) {{\scriptsize$3$}};

    \filldraw[thick] (0.12,1.00) circle(1.80pt);
    \filldraw[thick] (0.12,0.50) circle(1.80pt);
    \filldraw[thick] (0.12,0.00) circle(1.80pt);

    \node [anchor=west] (R1) at (1.44, 1.00) {{\scriptsize$1$}};
    \node [anchor=west] (R2) at (1.44, 0.50) {{\scriptsize$2$}};
    \node [anchor=west] (R3) at (1.44, 0.00) {{\scriptsize$3$}};
    \node [anchor=west] (R4) at (1.44, -0.50) {{\scriptsize$4$}};

    \filldraw[thick] (1.38,1.00) circle(1.80pt);
    \filldraw[thick] (1.38,0.50) circle(1.80pt);
    \filldraw[thick] (1.38,0.00) circle(1.80pt);
    \filldraw[thick] (1.38,-0.50) circle(1.80pt);

    \draw (0.06, 1.00) -- (1.44, 1.00);
    \draw (0.06, 1.00) -- (1.44, 0.50);
    \draw (0.06, 0.50) -- (1.44, 1.00);
    \draw (0.06, 0.50) -- (1.44, 0.00);
    \draw (0.06, 0.00) -- (1.44, -0.50);
\end{tikzpicture}}
\hspace{6.00ex}
\scalebox{1.20}{%
\begin{tikzpicture}
\pgfsetlinewidth{0.72pt}
    \node [anchor=east] (L1) at (0.06, 1.00) {{\scriptsize$1$}};
    \node [anchor=east] (L2) at (0.06, 0.50) {{\scriptsize$2$}};
    \node [anchor=east] (L3) at (0.06, 0.00) {{\scriptsize$3$}};

    \filldraw[thick] (0.12,1.00) circle(1.80pt);
    \filldraw[thick] (0.12,0.50) circle(1.80pt);
    \filldraw[thick] (0.12,0.00) circle(1.80pt);

    \node [anchor=west] (R1) at (1.44, 1.00) {{\scriptsize$1$}};
    \node [anchor=west] (R2) at (1.44, 0.50) {{\scriptsize$2$}};
    \node [anchor=west] (R3) at (1.44, 0.00) {{\scriptsize$3$}};
    \node [anchor=west] (R4) at (1.44, -0.50) {{\scriptsize$4$}};

    \filldraw[thick] (1.38,1.00) circle(1.80pt);
    \filldraw[thick] (1.38,0.50) circle(1.80pt);
    \filldraw[thick] (1.38,0.00) circle(1.80pt);
    \filldraw[thick] (1.38,-0.50) circle(1.80pt);

    \draw (0.06, 1.00) -- (1.44, 0.50);
    \draw (0.06, 0.50) -- (1.44, 1.00);
    \draw (0.06, 0.50) -- (1.44, 0.00);
    \draw (0.06, 0.00) -- (1.44, -0.50);
\end{tikzpicture}}
\hspace{6.00ex}
\scalebox{1.20}{%
\begin{tikzpicture}
\pgfsetlinewidth{0.72pt}
    \node [anchor=east] (L1) at (0.06, 1.00) {{\scriptsize$1$}};
    \node [anchor=east] (L2) at (0.06, 0.50) {{\scriptsize$2$}};
    \node [anchor=east] (L3) at (0.06, 0.00) {{\scriptsize$3$}};

    \filldraw[thick] (0.12,1.00) circle(1.80pt);
    \filldraw[thick] (0.12,0.50) circle(1.80pt);
    \filldraw[thick] (0.12,0.00) circle(1.80pt);

    \node [anchor=west] (R1) at (1.44, 1.00) {{\scriptsize$1$}};
    \node [anchor=west] (R2) at (1.44, 0.50) {{\scriptsize$2$}};
    \node [anchor=west] (R3) at (1.44, 0.00) {{\scriptsize$3$}};
    \node [anchor=west] (R4) at (1.44, -0.50) {{\scriptsize$4$}};

    \filldraw[thick] (1.38,1.00) circle(1.80pt);
    \filldraw[thick] (1.38,0.50) circle(1.80pt);
    \filldraw[thick] (1.38,0.00) circle(1.80pt);
    \filldraw[thick] (1.38,-0.50) circle(1.80pt);

    \draw (0.06, 1.00) -- (1.44, 1.00);
    \draw (0.06, 1.00) -- (1.44, 0.50);
    \draw (0.06, 0.50) -- (1.44, 0.00);
    \draw (0.06, 0.00) -- (1.44, -0.50);
\end{tikzpicture}}
$$

\vspace{-2.70cm}
\be{example-graphs}
\ee

\vspace{0.60cm}
\hfill$\Box$

\vspace{0.27cm}
Note that the requirement of no isolated right vertices in 
\Dref{domination-def} means that \emph{every} position $j \in [n]$~in 
the image of $\varphi$ is dominated by some position $i \in [m]$
of its domain. This 
is motivated by the application~to 
thermal-management coding for high-performance 
interconnects~\cite{CEKV18,CEKVW18}, briefly described in 
what follows.

\vspace{.50ex}
\subsection{Applications to coding for interconnects}\vspace{-.72ex}

Given an interconnect (bus) consisting of $n$ wires, it is desirable
to control its average power consumption. This can be achieved by making
sure that, during each synchronized transmission, the total number of
transitions on the $n$ wires is below a specified threshold $w$.
As shown in~\cite{CEKV18}, the above translates directly into
the requirement that the transmitted binary word $\yyy$ is in 
the Hamming ball $\cB(n,w)$. It is also desirable to control
the peak temperature of an interconnect by effectively cooling
its hottest wires. That is, given an arbitrary subset $S$ of $[n]$,
which 
corresponds to the positions of the hottest wires,
we wish to transmit a word $(y_1,y_2,\ldots,y_n)$ such that
$y_j = 0$ for all $j \in S$.
The \emph{cooling codes} of \cite{CEKV18} provide an elegant
solution to this problem using spreads or partial spreads~\cite{ES16,NS17}.
Domination mappings were introduced in \cite{CEKVW18} in order to
simultaneously achieve {both of these desirable properties}.
Given 
a subset $S$ of $[n]$,
one can follow the~edges~of the domination graph $G$ to find a subset
$S'$ of $[m]$ so that every position in $S$ is dominated by some
position in $S'$. Then, using spreads as described in \cite{CEKV18},
one can produce a binary word 
$\xxx = (x_1,x_2,\ldots,x_m)$ 
such that 
$x_i = 0$ for all $i \in S'$. If $\varphi$ is a $G$-dominating
mapping from $\bits^m$ into $\cB(n,w)$, the transmitted word
$\yyy = \varphi(\xxx)$ satisfies $\wt(\yyy) \le w$ and 
$y_j = 0$ for all $j \in S$. Since $\varphi$ is injective,
$\xxx$ can be recovered from $\yyy$ at the receiver.
For more details on this encoding/decoding procedure,
we refer the reader to~\cite{CEKV18,CEKVW18}.

\looseness=-1
In the context of coding for high-performance interconnects,
the parameters $m$ and $w$ would be usually given, and 
one would like to use
an $(m,n,w)$-domination mapping with the lowest
possible $n$, in order to mi\-nimize the redundancy $n-m$ of the
encoding procedure. 
We note that although domination mappings were defined in~\cite{CEKVW18},
they were not studied therein. In particular, it is not clear when
such mappings exist.

\vspace{1.80ex}
\subsection{Our contributions}\vspace{-.72ex}

Our primary objective in this paper is to answer the following
question: for which parameters $m$, $n$, $w$, 
do $(m,n,w)$-domination mapping exist?
A secondary objective is to study the structure of such mappings
and provide explicit constructions for them.
We believe that, 
aside from their applications to coding~for~interconnects, 
these problems are also interesting in their
own right, from the combinatorial perspective.

One of our main results is the following theorem, which shows
that if the domain $\bits^m$ is sufficiently large with respect
to $w$, then domination mappings \emph{always} exist.

\begin{theorem}
\label{theorem-main}
For every $w \ge 1$, there is a constant $m_0(w)$ with the following
property: for all $m \ge m_0(w)$, if 
$|\cB(n,w)| \ge 2^m$, then an $(m,n,w)$-domination mapping exists.
\end{theorem}

\looseness=-1
Thus, for sufficiently large $m$, whenever injections from $\bits^m$
into $\cB(n,w)$ exist at all, one of these injections is necessarily
an $(m,n,w)$-domination mapping. This is somewhat surprising; it is
certainly not true for small $m$. To prove \Tref{theorem-main}, we
introduce a certain bipartite graph with vertex set 
\mbox{$\bits^m \cup \cB(n,w)$},
which we call the \emph{compatibility graph}. We then use Hall's 
marriage theorem~\cite{Hall35} to show that the compatibility graph 
has a perfect matching, when $m$ is sufficiently large.
The fact that the conditions of Hall's marriage theorem 
are satisfied in this graph
is established in a~long and elaborate 
sequence of lemmas in Section\,\ref{sec:proof}.

\looseness=-1
The other contributions in this paper are
organized as follows.
We begin in the next section by establishing some
basic properties of domination graphs and domination
mappings. In particular, we consider certain restrictions
of a domination mapping $\varphi\!: \bits^m \to \cB(n,w)$ 
to a subset of the $m$ positions of its domain. We will show 
that all such restrictions are also domination mappings.
In Section\,\ref{sec:bounds}, we use some of these results
to derive necessary conditions for the existence of 
an $(m,n,w)$-domination mapping. In particular, we show
that \mbox{$n \ge 2m - w$} in any such mapping.
In Section\,\ref{sec:constructions}, we present several
explicit constructions of domination mappings. In particular,
we introduce the {product construction}; given any two
$(m_1,n_1,w_1)$-domination and $(m_2,n_2,w_2)$-domination mappings, 
this construction produces an 
$(m_1+m_2,n_1+n_2,w_1+w_2)$-domination mapping. We also
construct \emph{perfect} $(m,n,\kern-.5pt1)$-domination mappings 
with $n = 2^m\kern-1pt - 1$, 
so that \mbox{$|\cB(n,\kern-.5pt1)| = 2^m$}\linebreak
and $\varphi$ is a bijection. 
We furthermore describe another construction for $w = 2$ and $m$ odd;
it appears that~already in this case, the construction problem is
far from trivial.
In Section\,\ref{sec:algorithm}, we present 
a polynomial-time algorithm that, given any $m$, $n$, and $w$, 
determines whether an $(m,n,w)$-domination mapping exists 
for a~domination graph $G$ with an equitable degree distribution
(that is, the degree of all the left vertices in $G$ is either
$\trunc{n/m}$ or $\ceil{n/m}$).
To this end, we apply the K\"onig-Egev\'ary theorem~\cite{Konig}
to the bipartite compatibility graph in order to conclude that
the maximum matching and the \emph{maximum fractional matching}
in this graph have the same size. 
This leads to a linear program with an exponential number of
variables, but~we~use symmetries of the compatibility graph
to show that this problem can be solved in polynomial time.
%

One consequence of our constructions in Section\,\ref{sec:constructions}
is that the necessary conditions in Section\,\ref{sec:bounds} are also
sufficient for all $m \le 3w$. This resolves the existence problem
for $(m,n,w)$-domination mappings for small values of $m$. As 
a consequence of \Tref{theorem-main}, this problem is also resolved for 
large $m$. For intermediate values of $m$, we
can determine whether an $(m,n,w)$-domination mapping exists 
in polynomial time using the algorithm of Section\,\ref{sec:algorithm}
--- albeit only for an equitable domination graph.
We believe, however, that~the~decision produced by the polynomial-time
algorithm of Section\,\ref{sec:algorithm} holds unconditionally
(cf.~Conjecture\,\ref{Conjecture1}).

\vspace{2.70ex}
\section{Basic Properties of Domination Mappings}
\label{sec:properties}

\looseness=-1
We begin with some basic properties and notation for domination graphs.
Recall that, by definition, a domination graph is a bipartite graph
$G = \bigl([m] \cup [n], E\bigr)$ with no 
isolated vertices on the right side of~the~bipart\-ition.
The~fol\-lowing lemma shows that $G$ 
cannot have isolated vertices on the left side either. 

\begin{lemma}
\label{no-isolated}
Let $G = \bigl([m] \cup [n], E\bigr)$ be a domination graph.
If there is a vertex of degree zero in $[m]$, 
then a~$G$-domination mapping 
does not exist.
\vspace{-0.54ex}
\end{lemma}

\begin{proof}
\looseness=-1
Assume to the contrary that $\varphi\!: \bits^m \to \cB(n,w)$ 
is a~$G$-domination mapping. Since every posi\-tion in $[n]$ is
dominated by some position in $[m]$, we have 
$\varphi(\zero) = \zero$, where $\zero$ denotes 
the all-zero word of the appropriate length.
%
Suppose that the vertex $i \in [m]$ has degree zero, and 
let $\eee_i \in \bits^m$ be 
the~word~of weight one with the 
single nonzero in position $i$. Then $\varphi(\eee_i) = \zero$
since 
if $\deg(i) = 0$,
every posi\-tion in $[n]$ is dominated by a position in $[m]$
\emph{other than $i$}. 
Thus $\varphi(\eee_i) = \varphi(\zero)$,~con\-tradicting 
the fact that $\varphi$ is injective.\hspace*{-0.36ex}
\end{proof}

\looseness=-1
In view of \Lref{no-isolated}, we henceforth assume w.l.o.g.\ that
the degree of every vertex in a domination graph is at least one.
In fact, we can furthermore assume 
that the degree of all the vertices on the right side of 
the bipartition is \emph{exactly one}. 
Indeed, let $G = \bigl([m] \cup [n], E\bigr)$ be a domination graph, 
let $\varphi$ be a $G$-domination mapping, and suppose that
a vertex $j \in [n]$ has degree $\delta > 1$. We construct $G'$ from $G$
by removing,~arbitrar\-ily, some $\delta-1$ of the $\delta$ edges
incident upon 
$j$. Then, obviously, $\varphi$ is 
$G'$-dominating as well~(cf.~\Eref{Example1}).
It follows that any $(m,n,w)$-domination mapping is $G$-dominating
for some graph $G = \bigl([m] \cup [n], E\bigr)$~with $n$ edges and
$\deg(j) = 1$ for all $j \in [n]$.
Up to a permutation 
of $[n]$, any such domination graph $G$ is uniquely
described by the degrees $\delta_1,\delta_2,\ldots,\delta_m$
of its left vertices. We say that 
$(\delta_1,\delta_2,\ldots,\delta_m)$ is the 
\emph{\bfit degree sequence\linebreak of $G$}, and further assume
(w.l.o.g., up to a permutation of $[m]$) that
$\delta_1 \le \delta_2 \le\cdots\le\delta_m$.
With this assumption, a domination graph $G$ can be also
specified in terms of its \emph{\bfit degree distribution}
$(d_1,d_2,\ldots,d_\Delta)$, where $d_i$ is the number of left
vertices of degree $i$ and 
\smash{$\Delta \deff \max\{\delta_1,\delta_2,\ldots,\delta_m\}$}.
Clearly,
\be{degree-distribution}
d_1 + d_2 + \cdots + d_\Delta \, = \, m
\hspace{4.5ex}\text{and}\hspace{4.5ex}
d_1 + 2d_2 + \cdots + \Delta d_\Delta \, = \, n
\ee
A domination~graph is said to be
\emph{\bfit equitable} if
$\trunc{n/m} = \delta_1 \le \delta_2\le \cdots\le\delta_m = \ceil{n/m}$.
Equitable domination graphs play an important role in 
the study of domination mappings. In fact, we conjecture as follows.
\begin{conjecture}
\label{Conjecture1}
\looseness=-1
An $(m,n,w)$-domination mapping exists if and only if
there exists 
a mapping from $\{0,\kern-1pt1\}^m$ to $\cB(n,\kern-1pt w)$ 
that is $G$-dominating
for an equitable graph~$G$.
\end{conjecture}

We now consider certain basic properties of $(m,n,w)$-domination mappings.
The following easy lemma shows that the existence problem for such 
mappings is ``monotonic'' in both $n$ and $w$.
\begin{lemma}
\label{monotonic}
If 
an $(m,n,w)$-domination mapping exists, then there also exist
an $(m,n+1,w)$-domination mapping 
and an $(m,n,w+1)$-domination mapping.
\vspace{-0.54ex}
\end{lemma}

\begin{proof}
\looseness=-1
Given an $(m,n,w)$-domination map $\varphi$, we construct
an injective map $\varphi'\!: \bits^m \to \cB(n+1,w)$ 
as follows: $\varphi'(\xxx) = \bigl(\varphi(\xxx),0)$,
where $(\cdot,\cdot)$ stands for string concatenation.
Clearly, $\varphi'$ is an $(m,n+1,w)$-domination mapping.
The other claim 
follows trivially from
the fact that $\cB(n,w) \subseteq \cB(n,w+1)$.
\end{proof}

We next show that every $G$-domination mapping gives rise to
a multitude of 
\emph{derived domination mappings} corresponding to 
subgraphs of $G$ induced by 
subsets of $[m]$.

\begin{lemma}
\label{derived-maps}
Let $\varphi\!: \bits^m \to \cB(n,w)$ be a $G$-domin\-a\-tion
mapping. Fix an arbitrary left vertex $i$ of $G$, and let $G'$
be the induced subgraph of $G$ obtained~by~removing $i$ and
all 
vertices adjacent to $i$. Then $\varphi$
induces a $G'$-domination mapping $\varphi'$ from 
$\bits^{m-1}$ to\/ $\cB(n-\delta,w)$, where $\delta = \deg(i)$.
\vspace{-0.54ex}
\end{lemma}

\begin{proof}
\looseness=-1
$\!$The mapping $\varphi'\!: \bits^{m-1} \to \cB(n-\delta,w)$
is derived from $\varphi$ as follows. Given an arbitrary~word
$\xxx' = (x_1,x_2,\ldots,x_{m-1})$ in $\bits^{m-1}$,
we first extend it to 
$\xxx = (x_1,x_2,\ldots,x_{i-1},0,x_i,\ldots,x_{m-1}) \,{\in}\, \bits^m$.
Let $\{j,j+1,\ldots,j+\delta-1\} \,{\in}\, [n]$ be the set of positions dominated
by the vertex $i$ (where we have assumed w.l.o.g.\ that these positions are
consecutive). Then $\yyy = \varphi(\xxx)$ is of the form\vspace{-0.72ex}
$$
\yyy 
\, = \,
\bigl(y_1,y_2,\ldots,y_{j-1},
\underbrace{0,0,\ldots,0}_{\text{\small$\delta$~zeros}},
y_{j+\delta},\ldots,y_{n}),
\vspace{-1.26ex}
$$
and we set $\varphi'(\xxx') = (y_1,y_2,\ldots,y_{j-1},y_{j+\delta},\ldots,y_{n})$.
It is easy to see that the mapping $\varphi'$ thereby defined~is an
injection from $\bits^{m-1}$ to $\cB(n-\delta,w)$, 
and $\varphi'$ is $G'$-dominating.
\end{proof}

The construction of the derived mapping $\varphi'$ in \Lref{derived-maps}
is akin to the process of shortening linear codes in coding theory~\cite{MWS}. 
We restrict both the domain and the range of $\varphi$ to subsets thereof
consisting of those words that contain zeros in specified positions,
then puncture out these all-zero positions to obtain $\varphi'$.
Iterating this construction, a single $(m,n,w)$-domination mapping
gives rise to $2^m$ derived domination mappings (some of which
may be isomorphic or trivial) corresponding to all subsets of $[m]$.

An immediate consequence of \Lref{derived-maps} is that the existence 
problem for domination mappings is also monotonic in $m$. That is,
for all $m \ge 2$, if an $(m,n,w)$-domination mapping exists, we can use
\Lref{derived-maps} (in conjunction with \Lref{monotonic})
to construct an $(m-1,n,w)$-domination mapping.
In view of this~monotonic\-ity, it is natural to define the 
following parameters:\vspace{-0.54ex}
\begin{align*}
\mu(n,w) &~\deff~
\max\bigl\{m \in \Z ~:~ \text{an $(m,n,w)$-domination mapping exists} \bigr\}
\\[-0.27ex]
\nu(m,w) &~\deff~
\min\bigl\{n \in \Z ~:~ \text{an $(m,n,w)$-domination mapping exists} \bigr\}
\end{align*}
We say that an $(m,n,w)$-domination mapping is \emph{\bfit optimal}
if $n = \nu(m,w)$. 
%
We next
discuss necessary conditions for the existence
of an $(m,n,w)$-domination mapping, which can be regarded as bounds
on~$\nu(m,w)$. 

\vspace{3ex}
\section{Necessary Conditions}
\label{sec:bounds}

An $(m,n,w)$-domination mapping is injective by definition, so the 
size of its domain $\bits^m$ cannot~exceed the size of its range $\cB(n,w)$.
While trivial, this necessary condition is 
fundamental; we state it as follows.
\begin{lemma}
\label{lem:sum_cond}
For any $(m,n,w)$-domination mapping, we have
\be{cardinality-condition}
2^m 
\leq\: \sum_{j=0}^w \binom{n}{j}
\ee
\end{lemma}

\looseness=-1
In conjunction with the shortening procedure in \Lref{derived-maps},
the trivial necessary condition of \Lref{lem:sum_cond}~leads to a 
considerably more elaborate bound on the parameters 
of an $(m,n,w)$-domination mapping.
\begin{lemma}
\label{lemma6}
Let $G = \bigl([m] \cup [n], E\bigr)$ be a domination graph with
degree distribution $(d_1,d_2,\ldots,d_\Delta)$. Then for any
$G$-domination mapping $\varphi\!: \bits^m \to \cB(n,w)$, we have
\be{bound-general}
m 
~\le\:
\hspace{-.36ex}\min_{t_1,t_2,\ldots,t_\Delta}\hspace{-.36ex}
\left\{
(t_1\! + t_2 + \cdots + t_\Delta) \:+~ 
\log_2 \sum_{j=0}^w\! \binom{n - t_1-2t_2-\cdots-\Delta t_\Delta}{j}
\right\}
\ee
where the minimum is taken over all
nonnegative integers\/ $t_1,t_2,\ldots,t_\Delta$ with\,
$t_i \le d_i$\, for $i = 1,2,\ldots,\Delta$.
\vspace{-0.54ex}
\end{lemma}

\begin{proof}
Given $t_1,t_2,\ldots,t_\Delta$, we 
invoke \Lref{derived-maps}
iteratively to obtain a domination graph $G'$ with degree~distribution
$(d_1\kern-1pt-t_1,d_2\kern-1pt-t_2,\ldots,d_\Delta\kern-1pt-t_\Delta)$ 
and a $G'$-domination mapping $\varphi'$ with parameters 
$(m'\kern-1pt,n'\kern-1pt,w)$, 
where\linebreak
$m' = m - (t_1\kern-1pt + t_2\kern-1pt + \cdots + t_\Delta)$
and 
$n' = n - t_1-2t_2-\cdots-\Delta t_\Delta$.
Applying \Lref{lem:sum_cond} to $\varphi'$
establishes~\eq{bound-general}.\linebreak 
\end{proof}

\vspace{-1.26ex}
Observe that the minimization in \eq{bound-general} can be reduced
to the minimum of at most $m$ terms, parametrized by 
$s = t_1\! + t_2 + \cdots + t_\Delta$ with $s \le m-1$.
The following lemma is an important special case of \Lref{lemma6}.
This lemma shows, in particular, that the trivial necessary condition
of \Lref{lem:sum_cond} is, in general, not sufficient.\vspace{0.36ex}

\begin{lemma}
\label{lem:tight}
For any $(m,n,w)$-domination mapping, we have $n \ge 2m - w$.
\vspace{-0.54ex}
\end{lemma}

\begin{proof}
Set $t_1 = 0$ and $t_i = d_i$ for $i = 2,3,\ldots,\Delta$ in \Lref{lemma6}.
Then 
$n - t_1-2t_2-\cdots-\Delta t_\Delta = d_1$ 
while
$t_1\kern-1pt + t_2\kern-1pt + \cdots + t_\Delta = m - d_1$
in view of~\eq{degree-distribution}, and 
the bound in \eq{bound-general} reduces to
$$
d_1 
 \le\, 
\log_2 \sum_{j=0}^w\! \binom{d_1}{j}
$$
This is only possible if $w \ge d_1$. Now, again 
in view of~\eq{degree-distribution},
we have
$
n \ge d_1 + 2(d_2 + \cdots + d_\Delta) = 2m - d_1
$,
and the lemma follows.
\end{proof}

\looseness=-1
The bounds on $\nu(m,w)$
resulting from \Lref{lem:sum_cond} and \Lref{lem:tight} 
are 
illustrated 
in Figure\,1 for 
\mbox{$w = 3$}, 
and we can see that neither of them implies the other.
In fact, for all $w$, it 
can be readily shown 
that \Lref{lem:tight} is\linebreak
tighter than \Lref{lem:sum_cond} whenever $w < m \le 3w$,
whereas \Lref{lem:sum_cond} is tighter when $m$ is large.

\looseness=-1
We next~investigate the conditions under which equality can
hold in the bounds of \Lref{lem:sum_cond} and \Lref{lem:tight}.\linebreak
A domination mapping 
that establishes a \emph{bijection} between
$\bits^m$ and $\cB(n,\kern-1ptw)$ will be 
called~\emph{\bfit perfect};~clear\-ly
an $(m,n,w)$-domination mapping is perfect if and
only if its parameters satisfy
$2^m = \sum_{j=0}^w \binom{n}{j}$.
There~are\linebreak
only four known cases where (an incomplete) sum of binomial coefficients
gives a power of $2$. The $(m,n,w)$ triples for these four cases
are given by
\be{perfect-cases}
(m,m+1,m/2)~\text{for $m$ even},
\hspace{1.80ex}
(m,2^m-1,1),
\hspace{1.80ex}
(12,90,2),
\hspace{1.80ex}
(11,23,3)
\ee
All these cases were discovered in the context of perfect 
error-correcting codes~\cite{Heden,Solovieva} in the 1960s, 
and~none were found since. It is well known that perfect binary
codes --- that is, partitions of $\bits^n$ into translates of 
the Hamming ball
$\cB(n,w)$ --- exist in only three of the four cases (it is
not possible to partition $\bits^{90}$ into Hamming balls).
Remarkably, we have found perfect domination mappings for
each of the four triples~listed in~\eq{perfect-cases}.
The first set of parameters is attained for $m = 2$
by the $(2,3,1)$-domination mapping with degree~se\-quence $(1,2)$,
given by
$$
\varphi(00) \,=\, 000,~~~~~
\varphi(01) \,=\, 010,~~~~~
\varphi(10) \,=\, 100,~~~~~
\varphi(11) \,=\, 001
$$
By \Lref{lem:tight}, 
there are no $(m,m+1,m/2)$-domination mappings for $m \ge 4$.
In Section\,\ref{sec:w=1}, we 
explicitly construct
perfect $(m,2^m\kern-1pt-1,1)$-domination mappings for all $m \ge 1$.
Finally, perfect domination mappings with parameters $(12,90,2)$
and $(11,23,3)$ were found by computer search using the methods
of Section\,\ref{sec:algorithm}.

\newpage
$\,$\\[-1.62cm]

\hspace*{0.45cm}%
\begin{minipage}{9.00cm}
\hspace*{0.81cm}{\large $n$}\\[0.09cm]
%
\mbox{\includegraphics[width=7.95cm]{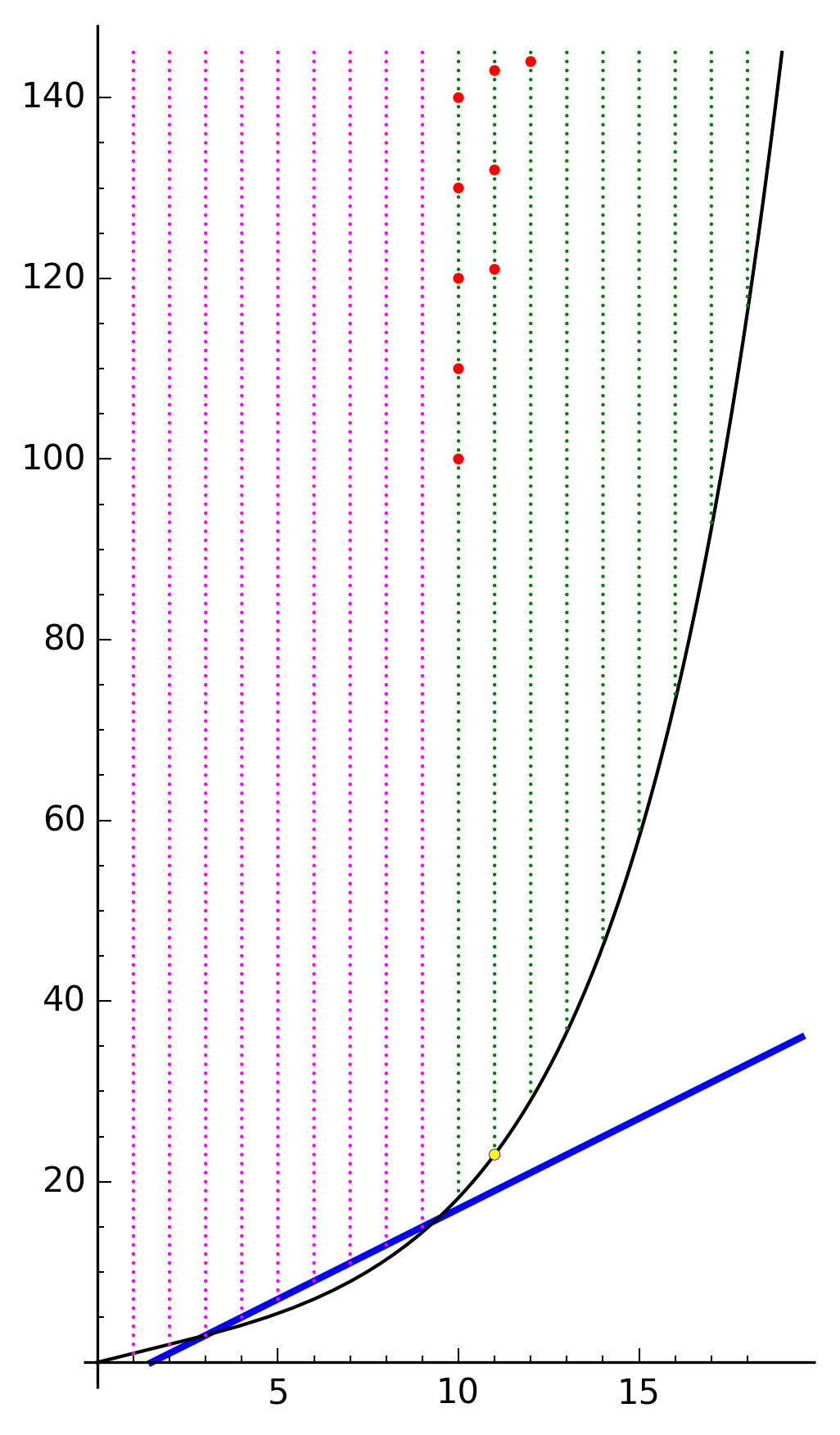}}%
%
%
\end{minipage}
\hspace*{-1.80cm}\raisebox{-1.00cm}{
\small
\begin{tabular}{c@{~~~}l}
\raisebox{-.06cm}%
{\includegraphics[width=0.54cm,height=0.27cm]{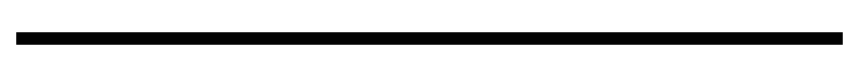}} &
necessary condition of \Lref{lem:sum_cond}
\\
\raisebox{-.06cm}%
{\includegraphics[width=0.54cm,height=0.27cm]{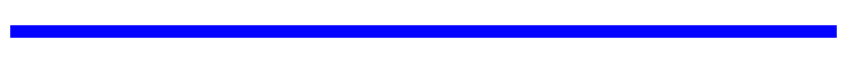}} &
necessary condition of \Lref{lem:tight}
\\[0.18cm]
\includegraphics[height=0.18cm]{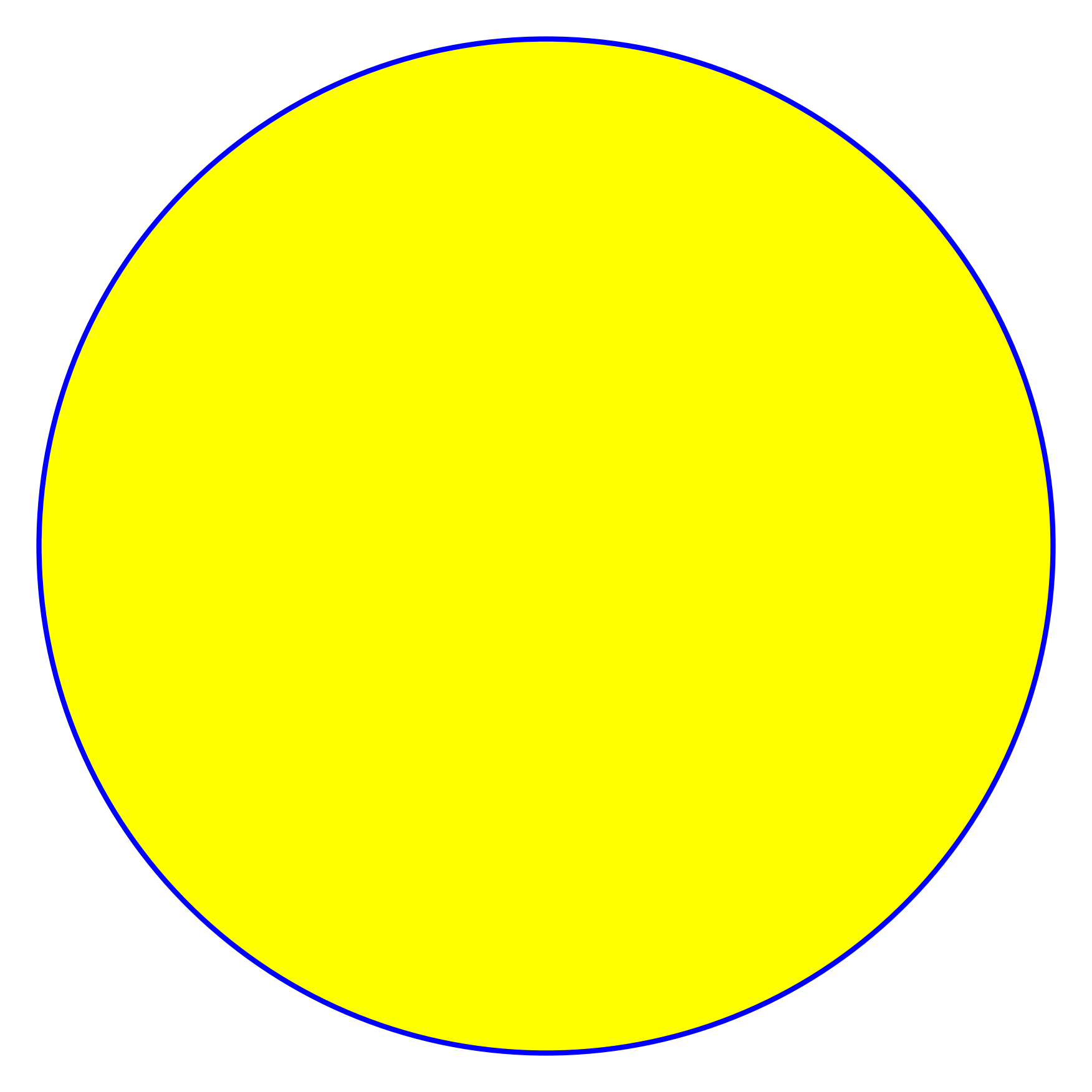} &
perfect domination mapping
\\
\includegraphics[height=0.09cm]{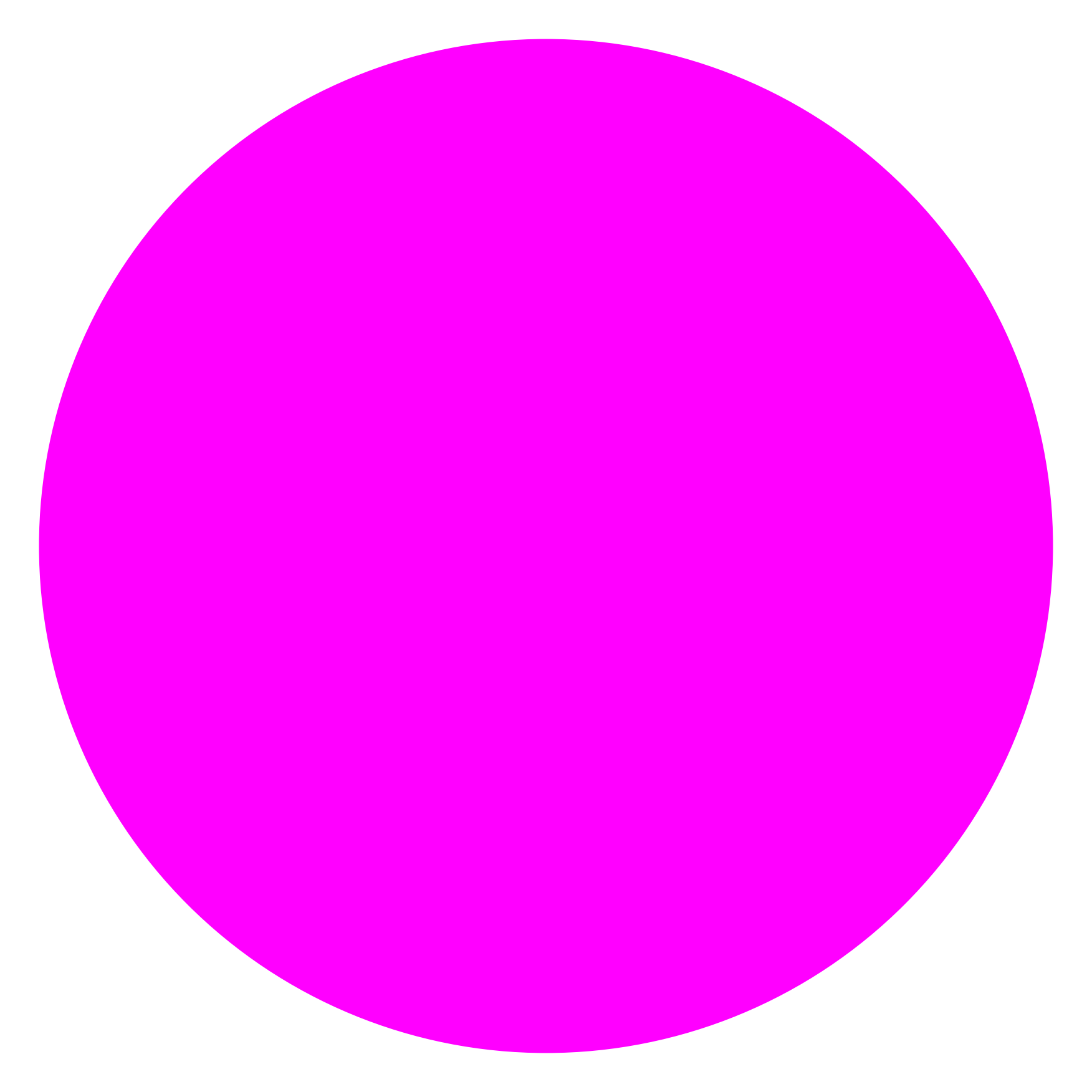} &
mapping constructed in Sections \ref{sec:properties} and \ref{sec:constructions}
\\
\includegraphics[height=0.09cm]{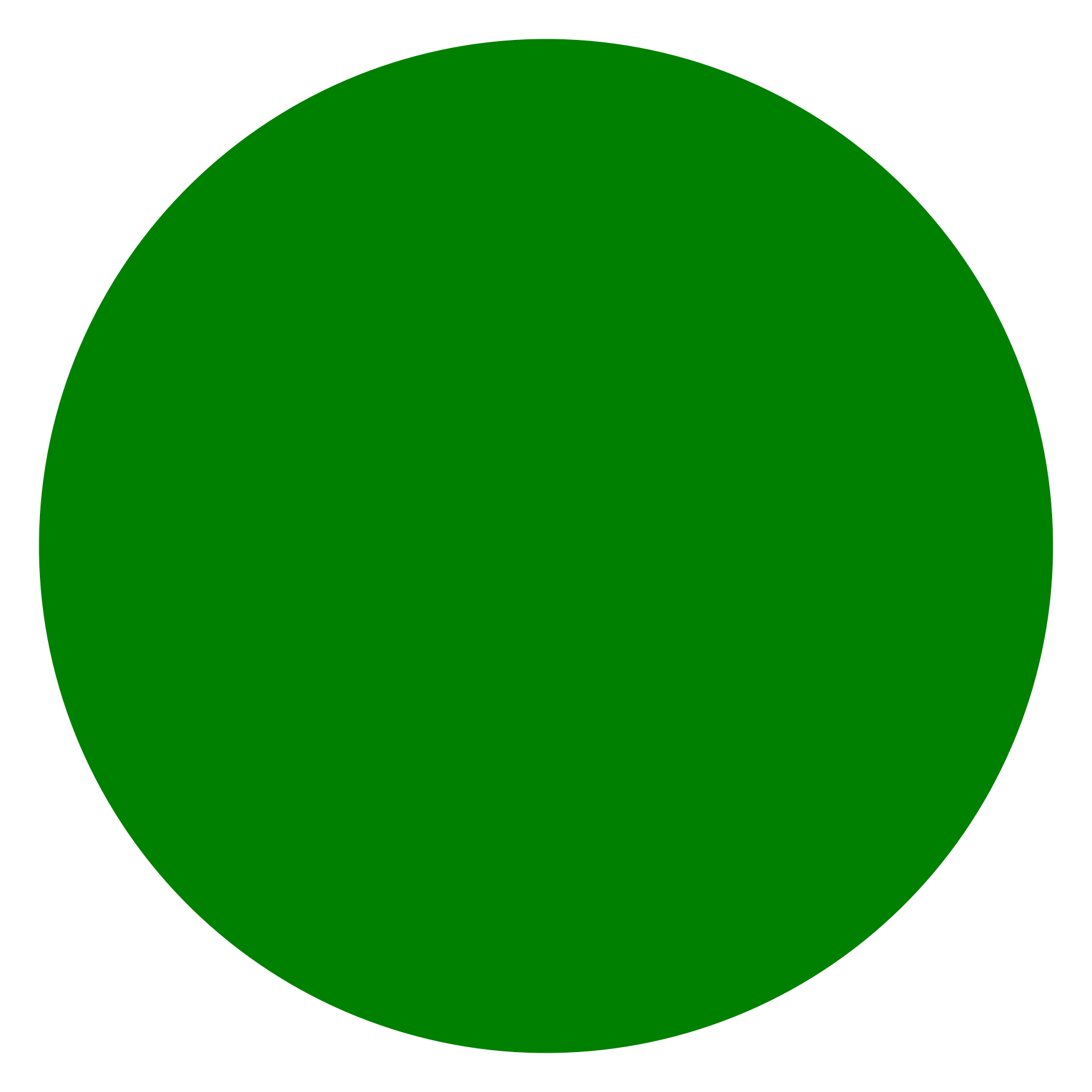} &
mapping exists by the results of Section\,\ref{sec:properties} 
\\
\includegraphics[height=0.18cm]{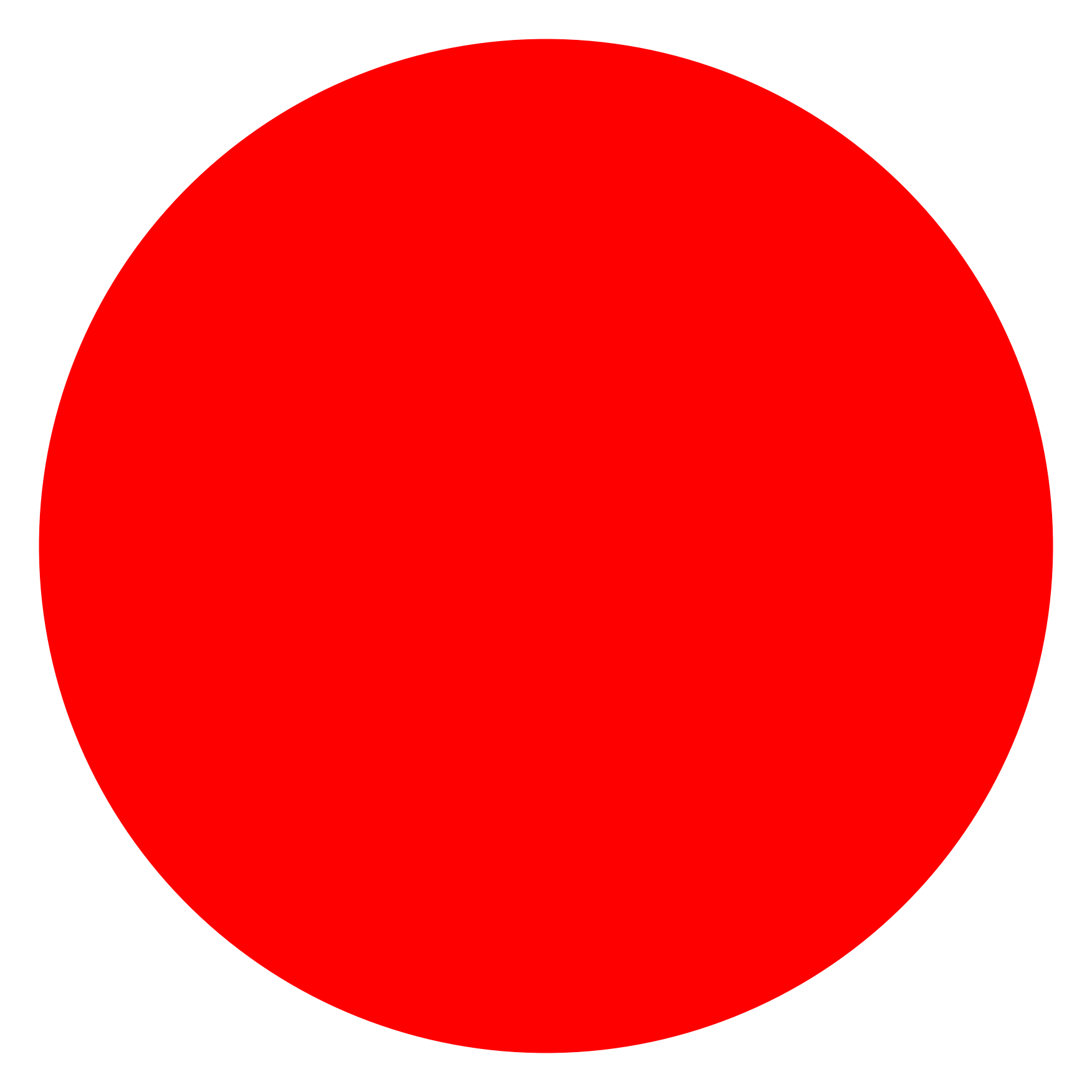} &
mapping exists by the results of Section\,\ref{sec:proof}
\end{tabular}}

\vspace{-0.90cm}\hspace*{8.55cm}{\large $m$}

\vspace{0.72cm}
\centerline{%
\textbf{Figure\,1.}
\textsl{Bounds, constructions, and existence 
of $(m,n,w)$-domination mappings for $w=3$.}}

\vspace{0.63cm}
In contrast to equality in \eq{cardinality-condition}, which is 
quite rare, equality in \Lref{lem:tight} can be achieved much more
easily. The next lemma gives necessary and sufficient conditions 
for this.

\begin{lemma}
\label{lem:optimal_domination}
Let $\varphi$ be an $(m,n,w)$-domination mapping. Then $n = 2m - w$
if and only if the degree distribution of 
the underlying domination graph is 
$d_1 = w$,\, $d_2 = m-w$,\, and\, $d_i = 0$\kern2pt\ for $i \ge 2$.
\vspace{-0.54ex}
\end{lemma}

\begin{proof}
It suffices to show that the conditions on the degree distribution
are necessary, since the fact that they are sufficient is trivial
from~\eq{degree-distribution}. Suppose $n = 2m-w$. Then 
\eq{degree-distribution} implies that
$$
2m - w
\,\ge\,
d_1 + 2d_2 + 3(d_3 + d_4 + \cdots + d_\Delta)
\,=\,
d_1 + 2d_2 + 3(m - d_1 - d_2)
\,=\,
3m - 2d_1 - d_2
$$
which simplifies to $2d_1 + d_2 \ge m + w$. Now, it follows from
(the proof of) \Lref{lem:tight} that $d_1 \le w$ in \emph{any}\linebreak
$(m,n,w)$-domination mapping. Hence $2d_1 + d_2 \ge m + w$
further implies 
$
d_1 + d_2 
\ge
m
$.
But 
this~is~only~pos\-sible 
if $d_1 + d_2 = m$ and $d_3 + d_4 + \cdots + d_\Delta = 0$.
Finally, if $d_1\! < w$, then $n 
\ge 2m - d_1 > 2m-w$.
\end{proof}

In the next section, for all $m$ in the range $w \le m \le 3w$,
we present an explicit construction of domination mappings whose
degree distribution satisfies the conditions of \Lref{lem:optimal_domination}.
By \Lref{lem:tight}, such mappings are optimal, and
therefore $\nu(m,w) = 2m - w$ when $w \le m \le 3w$ and $w \ge 3$.

Finally, 
we note that the case $m \le w$ is trivial. In this case,
\Lref{lem:sum_cond} reduces to $n \kern-1pt\ge m$. Thus~\mbox{$\nu(m,w) = m$}
and the optimal $(m,m,w)$-domination mapping is the identity
map from $\bits^m$ to itself.\vspace*{-1.50ex}

\newpage

\vspace{3ex}
\section{Constructions}
\label{sec:constructions}
\vspace{-0.56ex}

\looseness=-1
In this section, we present several explicit \emph{constructions} of
(optimal) domination mappings. These results are distinct from our
results in Sections \ref{sec:algorithm} and \ref{sec:proof}, which
are concerned solely with the \emph{existence} of such mappings.

\vspace{.90ex}
\subsection{Product construction}\vspace{-.72ex}
\label{sec:product}

We begin with a simple and effective recursive construction that
combines any two domination mappings~$\varphi_1$ and $\varphi_2$ 
to produce another domination mapping $\varphi$.
Notably, if the mappings $\varphi_1$ and~$\varphi_2$ 
attain the bound of \Lref{lem:tight}, then so does
the mapping $\varphi$ obtained from this construction.

Let 
$\varphi_1\!: \bits^{m_1} \to \cB(n_1,w_1)$ 
and
$\varphi_2\!: \bits^{m_2} \to \cB(n_2,w_2)$ 
be arbitrary domination mappings. Then their\kern1pt\
\emph{\bfit product}\kern1pt\
$\varphi = \varphi_1 \times \varphi_2$\, is a mapping from 
$\bits^{m_1+m_2}$ into $\cB(n_1\kern-1pt+n_2,w_1\kern-1pt+w_2)$ defined as 
follows:
\be{product-def}
\varphi(\xxx_1,\xxx_2)
\ = \
\bigl(\varphi_1(\xxx_1),\varphi_2(\xxx_2)\bigr)
\ee
where $\xxx_1\,{\in}\,\bits^{m_1}$, $\xxx_2\,{\in}\,\bits^{m_2}$,
and $(\cdot,\cdot)$ stands for string concatenation. 
That is, in order to find the~image of a word $\xxx \in \bits^{m_1+m_2}$
under $\varphi$, we first parse $\xxx$ as $(\xxx_1,\xxx_2)$, then
apply $\varphi_1$ and $\varphi_2$ to the two parts. 
\begin{theorem}
\label{thm:product}
If $\varphi_1$ is an $(m_1,n_1,w_1)$-domination mapping 
and $\varphi_2$ is an $(m_2,n_2,w_2)$-domination mapping,
then their product $\varphi = \varphi_1 \times \varphi_2$\kern1pt\ is 
an $(m_1\kern-1pt+ m_2, n_1\kern-1pt+n_2,w_1\kern-1pt+w_2)$-domination mapping.
\vspace{-0.54ex}
\end{theorem}

\begin{proof}
\looseness=-1
The parameters of $\varphi$ and the fact that $\varphi$ is injective
are obvious from \eq{product-def}. Thus it remains to show~that 
every position $j \in [n_1\kern-1pt+n_2]$ is dominated by some
position $i \in [m_1\kern-1pt+m_2]$. But this is easy. Let $G_1$
and $G_2$ be the domination graphs for $\varphi_1$ and $\varphi_2$,
respectively. If $j \le n_1$, we follow the corresponding edge of $G_1$
to find $i$. If $j > n_1$, we follow the edge of $G_2$ that corresponds
to $j - n_1$ to find $i^*$, then set $i = i^* + m_1$.
\end{proof}

Observe that the proof of \Tref{thm:product} relates the degree 
distribution $(d_1,d_2,\ldots,d_\Delta)$ of (the domination graph
for) the product 
$\varphi = \varphi_1 \times \varphi_2$
to the degree distributions $(d'_1,d'_2,\ldots,d'_{\Delta'})$
and $(d^*_1,d^*_2,\ldots,d^*_{\Delta^*})$ of the 
constituent mappings $\varphi_1$ and $\varphi_2$. 
Clearly $\Delta = \max\{\Delta',\Delta^*\}$ and $d_i = d'_i + d^*_i$
for all $i$. 

%

\begin{corollary}
\label{cor:product_optimal1}
For all $w\ge 3$, the product construction 
yields a $(3w,5w,w)$-domination mapping.
\vspace{-0.72ex}
\end{corollary}

\begin{proof}
Using computer search, we found domination mappings 
$\varphi_3$, $\varphi_4$, and $\varphi_5$ with parameters
$(9,15,3)$, $(12,20,4)$, and $(15,25,5)$, respectively.
Since any $w \ge 3$ is an integer linear combination of
$3,4$, and $5$, a~$(3w,5w,w)$-domination mapping
can be constructed as the appropriate product of 
$\varphi_3$, $\varphi_4$, and $\varphi_5$.
\end{proof}

\begin{theorem}
$$
\nu(m,w) 
\, = \,
2m - w
\hspace{6.0ex}\text{for all~ $w \ge 3$~ and~ $w \le m \le 3w$}
$$
\end{theorem}

\begin{proof}
In view of \Lref{lem:tight}, it suffices to show that an 
$(m,2m-w,w)$-domination mapping exists for the specified ranges
of $m$ and $w$. We begin with the $(3w,5w,w)$ mapping constructed
in \Cref{cor:product_optimal1}. Note that the degree distribution
of this mapping is 
$(w,2w)$ by \Lref{lem:optimal_domination}.
Thus for all $i = 1,2,\ldots,2w$, we can invoke \Lref{derived-maps}
to remove $i$ vertices of degree $2$, thereby producing
a $(3w-i,5w-2i,w)$-domination mapping.\vspace*{-0.36ex}

\end{proof}

\vspace{-2.00ex}
\looseness=-1
We observe that optimal domination mappings that attain the
bound of \Lref{lem:tight} also exist beyond the $m = 3w$ threshold.
For example, there exists an optimal $(13,22,4)$-domination mapping with 
degree distribution $(4,9)$.
Taking this 
mapping as both
$\varphi_1$ and $\varphi_2$ in the product
construction, we obtain a $(26,44,8)$-domination mapping
with degree distribution $(8,18)$. This 
mapping is again optimal by \Lref{lem:optimal_domination}.

\vspace{1.00ex}
The product construction further implies the following immediate 
upper bound on $\nu(m,w)$.
\begin{lemma}
\label{lem:Prod_n_bound}
If $\nu(m_1,w_1) \leq n_1$ and $\nu(m_2,w_2) \leq n_2$, then $\nu(m_1+m_2,w_1+w_2) \leq n_1 + n_2$.
\end{lemma}

\newpage
Using Lemma~\ref{lem:Prod_n_bound}, we can find an upper bound on $\nu(m,w)$ for all admissible parameters.
We can write $m=m_1+m_2$ and $w=w_1+w_2$ and use the
related upper bounds on $\nu(m_1,w_1)$ and $\nu(m_2,w_2)$ and apply Lemma~\ref{lem:Prod_n_bound}.
Hence, we can obtain recursively upper bounds on $\nu(m,w)$ 
for any triple $(m,n,w)$.

\subsection{Domination mappings into the Hamming ball of small radius}
\label{sec:w=1}

Perfect domination mappings for $w=1$ are easily constructed.
By \Lref{lem:sum_cond}, we have 
$\nu (m,1) \geq 2^m -1$ for $w=1$.
For a given $m \geq 1$, we will construct a
perfect $(m,2^m-1,1)$-domination mapping with degree sequence
$(\delta_1,\delta_2,\ldots,\delta_m)$, 
where $\delta_i=2^i$ for all $i$.
It should be noted that we can shorten this domination mapping to obtain
a perfect $(m',2^{m'}-1,1)$-domination mapping for any $m' < m$.
The $(m,2^m-1,1)$-domination mapping $\varphi$ is quite simple;
in fact, many such mappings exist.
Clearly, $\varphi(\zero) = \zero$.
Now, for an integer $j \in [2^m-1]$, let $\bbb(j)$ be the
binary word of length $m$ which forms the binary representation~of~$j$.
We then define $\varphi(\bbb(j)) = \yyy$, where $\yyy$ 
is the binary word of length $2^m-1$ and weight one, with
the single nonzero in its $j$-th position.
The proof of  correctness for
this construction can be given by induction on~$m$ 
or a related argument. This is summarized with the following theorem.

\begin{theorem}
\label{thm:w=1}
There exists a perfect $(m,2^m-1,1)$-domination mapping 
for all $m \geq 1$. Hence
$\nu(m,1)=2^m -1$ and $\mu(n,1) = \lceil \log_2 (n+1)\emph{} \rceil$.
\end{theorem}

A construction of domination mappings for $w=2$ is given in 
Appendix\,\ref{appendix-A}. This construction attains~the 
bound of \Lref{lem:sum_cond} when $m$ is
odd. In fact, using these results 
and Section\,\ref{sec:proof}, we determine $\nu(m,2)$~for~all~$m$.\vspace{2.70ex}

\section{Polynomial-Time Algorithm}
\label{sec:algorithm}

In this section, we present a polynomial-time algorithm that, given any $m$, $n$, and $w$, 
determines whether an $(m,n,w)$-domination mapping exists for a domination graph $G$ with an equitable degree distribution. 
To do so, we introduce a certain bipartite graph associated to $G$,  
which we call the \emph{compatibility graph}. 

\begin{definition}
\label{definition-compatibility}
Given a domination graph $G=([m]\cup[n],E)$, the {\bfit compatibility graph} defined by $G$ is the bipartite graph whose vertex set is $\{0,1\}^m\cup \cB(n,w)$. There is an edge from $\xxx\in \{0,1\}^m$ to $\yyy\in \cB(n,w)$ 
if and only if for $(i,j) \in E$, we have that $x_i = 0$ implies $y_j = 0$.
\end{definition}

Fix a domination graph $G=([m]\cup[n],E)$.
For brevity, we write $\cS\deff \{0,1\}^m$ and $\cR\deff \cB(n,w)$.
Also, for $\xxx\in \cS$ and $\yyy\in \cR$, we say that $\xxx$ {\em dominates} $\yyy$ if and only if 
for $(i,j) \in E$, we have that $x_i = 0$ implies $y_j = 0$. 
Hence, the compatibility graph defined by $G$ has vertex set $\cS\cup\cR$ and 
$\xxx\in \cS$ is adjacent to $\yyy\in \cR$ if and only if $\xxx$ dominates $\yyy$.
The next theorem states that the existence of a $G$-domination mapping is equivalent to a certain graph theoretic property of the compatibility graph defined by $G$.

\begin{theorem}
\label{thm:sub_dominate}
Let $H$ be the compatibility graph defined by  domination graph $G=([m]\cup [n],E)$.
There exists a $G$-domination mapping if and only if there exists a subgraph of $H$ such the degree of each vertex in $\cS$ is exactly one and the degree of each vertex in $\cR$ is at most one.
\end{theorem}

\begin{proof}
Let $\varphi:\cS\to \cR$ be a $G$-dominating mapping. We define the subgraph $H'=(\cS\cup\cR,E')$ where
$E'=\{(\xxx,\varphi(\xxx)): \xxx\in \cS\}$. Since $\xxx$ dominates $\varphi(\xxx)$, $H'$ is a subgraph of $H$.
Clearly, the degree of every vertex in $\cS$ is exactly one. Since $\varphi$ is an injection, the degree of every vertex in $\cR$ is at most one.

Conversely, suppose that there exists such a subgraph $H'=(\cS\cup\cR,E')$.
For each $\xxx\in \cS$, define $\varphi(\xxx)$ to be the unique vertex in $\cR$ adjacent to $\xxx$.
Then it is readily verified that $\varphi$ is a $G$-domination mapping.
\end{proof}

%
%
%
%
%
%
%
%

Therefore, following Theorem~\ref{thm:sub_dominate}, our task is reduced to determining the existence of a perfect matching in the bipartite compatibility graph defined by $G$. 
To this end, we recall the famous K\"onig-Egev\'ary theorem~\cite{Konig} for general graphs.
Let $\cG=(V,E)$ be a graph and we define the following quantities via certain integer optimization problems.
\begin{align}
M(\cG)
&\deff \max \left\{ \sum_{e\in E} X_e : 
X_e\in \{0,1\} \mbox{ for all }e\in E, \sum_{e \mbox{ incident to } v } X_e\le 1 \mbox{ for all }v\in V \right\},\label{matching}\\
C(\cG)
&\deff \min \left\{ \sum_{v\in V} Y_v : 
Y_v\in \{0,1\} \mbox{ for all }v\in V, \sum_{v \mbox{ incident to  } e } Y_v\ge 1 \mbox{ for all }e\in E \right\}.\label{cover}
\end{align}
Then $M(\cG)$ and $C(\cG)$ correspond to the sizes of the {\em maximum matching} and {\em minimum vertex cover}, respectively. Via weak duality, we have that $M(G)\le C(G)$. 
Next, we consider the {\em relaxed} or {\em fractional} versions of the optimization problems.
\begin{align}
M^*(\cG)
&\deff \max \left\{ \sum_{e\in E} X_e : 
0\le X_e\le 1 \mbox{ for all }e\in E, \sum_{e \mbox{ incident to } v } X_e\le 1 \mbox{ for all }v\in V \right\},\label{matching-relaxed}\\
C^*(\cG)
&\deff \min \left\{ \sum_{v\in V} Y_v : 
0\le Y_v\le 1 \mbox{ for all }v\in V, \sum_{v \mbox{ incident to } e } Y_v\ge1 \mbox{ for all }e\in E \right\}.\label{cover-relaxed}
\end{align}
We refer to $M^*(\cG)$ and $C^*(\cG)$ as the {\em maximum fractional matching} and {\em minimum fractional vertex cover}, respectively. Then we have the following inequality,
$M(\cG)\le M^*(\cG)\le C^*(\cG)\le C(\cG)$.

For bipartite graphs, K\"onig-Egev\'ary theorem~\cite{Konig} states that $M(G)=C(G)$. 
In other words, the maximum matching and the maximum fractional matching in a bipartite graph have the same size. 
Specifically,
\[M(\cG)= M^*(\cG)= C^*(\cG)= C(\cG).\]

\subsection{Maximum matching in compatibility graphs}

Recall that the vertex set in our compatibility graph is $V=\cS\cup \cR$ while 
the edge set is $E=\{(\uuu,\vvv) \in \cS \times \cR: \uuu \mbox{ dominates } \vvv\}$.
Then the size of maximum matching is given by the following linear program
\begin{equation}\label{comp:lp}
\max \left\{\sum_{e\in E}X_e : \bA\bX\le \vone,\bX\ge \vzero \right\},
\end{equation}
where $\bA$ is a matrix whose rows are indexed by $V$ and columns are indexed by $E$.
Here, $\bA(v,e)=1$ if $v$ is incident to $e$, and $\bA(v,e)=0$, otherwise.

\begin{example}\label{exa:A}
Set $m=2$, $n=4$, and $w=1$.
Then $\cR=\{00,01,10,11\}$ and $\cS=\{0000,0001,0010,0100,1000\}$.
Here, $V=\cS\cup \cR$, and $E=\{e_i: i\in [12]\}$, where 
\begin{align*}
e_1 &\deff (00,0000),\\
e_2 &\deff (01,0000),&
e_3 &\deff (01,0001),&
e_4 &\deff (01,0010),\\
e_5 &\deff (10,0000),&
e_6 &\deff (10,0100),&
e_7 &\deff (10,1000),\\
e_8 &\deff (11,0000),&
e_9 &\deff (11,0001),&
e_{10} &\deff (11,0010),&
e_{11} &\deff (11,0100),&
e_{12} &\deff (11,1000).
\end{align*}
Then the matrix $\bA$ is given by 
\[
\bA=\left(
\begin{array}{c ccc ccc ccccc}
1& 0&0&0& 0&0&0& 0&0&0&0&0\\
0& 1&1&1& 0&0&0& 0&0&0&0&0\\
0& 0&0&0& 1&1&1& 0&0&0&0&0\\
0& 0&0&0& 0&0&0& 1&1&1&1&1\\
1& 1&0&0& 1&0&0& 1&0&0&0&0\\
0& 0&1&0& 0&0&0& 0&1&0&0&0\\
0& 0&0&1& 0&0&0& 0&0&1&0&0\\
0& 0&0&0& 0&1&0& 0&0&0&1&0\\
0& 0&0&0& 0&0&1& 0&0&0&0&1
\end{array}
\right)\, .
\]
Set
\[\bX=\left(1,0,\frac12,\frac12, 0,\frac12,\frac12, 0, \frac14,\frac14,\frac14,\frac14\right).\]
We check that $\bX$ is indeed a feasible vector whose objective value is four. 
Since the size of a matching in the compatibility graph is at most four, 
we find that size of maximum matching in this graph is four.
\end{example}

Unfortunately, the number of variables in the program \eqref{comp:lp} is 
given by $|E|=\Theta(2^{2m-1})$ when $m\ge 2w$.
Indeed, for $\uuu\in \{0,1\}^m$, when $\wt(\uuu)\ge w$, $\uuu$ dominates all words in $\cB(n,w)$. 
Since $|\cR|=|\cB(n,w)|\ge |\{0,1\}^m|=2^m$,
we have that $|E|\ge (\sum_{j=w}^m \binom{m}{j})|\cR|\ge 2^{m-1}\cdot 2^m=2^{2m-1}$.
In other words, to determine the existence of a mapping by solving this linear programme 
requires time {\em exponential} in $m$.
In what follows, we reduce the running time to {\em polynomial} in $m$ and $w$.

In do so, we notice that many entries in the optimal solution of Example~\ref{exa:A} are identical.
This hints that we may reduce the number of variables in the program and 
we do so by exploiting certain symmetries of the linear program.

\subsection{Symmetries of the linear program}

Let $M$ and $N$ be integers and let $A$ be an $M\times N$ matrix.
Consider a linear program of the form
\[\max \left\{\sum_{i\in [N]}x_i : A\xxx\le \vone,\xxx\ge \vzero \right\}.\]

Let $\pi:[N]\to [N]$ be a permutation on the set of $N$ variables. 
For any vector $\xxx\in \mathbb{R}^N$, let $\xxx^\pi$ denote the vector whose $i$th component is given by $x_{\pi(i)}$.
Let $P_\pi$ denote the binary matrix that represents the permutation $\pi$ and 
hence, by definition, we have that $P_\pi \xxx=\xxx^\pi$.

A permutation on the $N$ variables can be regarded as a permutation on the $N$ columns of $A$.
We also consider a permutation $\pi_{\rm row}$ on the $M$ rows of $A$ and
similarly let $P_{\pi_{\rm row}}$ denote the binary matrix that represents the permutation $\pi_{\rm row}$.

\begin{definition}
A permutation $\pi: [N]\to [N]$ is {\bfit $A$-preserving}
if $P_{\pi_{\rm row}}A P_\pi=A$ for some permutation $\pi_{\rm row}: [M] \to [M]$.
\end{definition}

This definition of $A$-preserving permutations can be found in linear programming literature 
(see Margot \cite{Margot:2003}, and B{\H o}di and Herr \cite{Bodi:2009}),
where the authors exploit symmetries to reduce the dimension of their linear programs.
In particular, B{\H o}di and Herr demonstrated Proposition \ref{prop:reduction} in a more general setting.
Independently, Fazelli {\em et al.} obtained Proposition \ref{prop:reduction} in the specialized setting of finding a fractional transversal in hypergraphs. For completeness, we rederive Proposition \ref{prop:reduction} in 
Appendix~\ref{app:reduction}.

Let $G_A$ denote a subgroup of the group of all $A$-preserving permutations and let $G_A$ act on $[N]$.
Suppose that the collection of orbits under this action is $\cO=\{O_1, O_2,\ldots, O_n\}$.
A vector $\xxx$ is defined to be {\em $\cO$-regular} if for all $k\le N$, $x_i=x_j$ for all $i,j\in O_k$.

\begin{proposition}\label{prop:reduction}
Suppose that $A\xxx\le \vone$ and $\sum_{i=1}^N x_i=\lambda$.
Then there exists an $\cO$-regular vector $\xxx^*$ such that 
$A\xxx^\pi\le \vone$ and $\sum_{i=1}^N x_i^*=\lambda$.
\end{proposition}

Therefore, applying Proposition~\ref{prop:reduction}, we establish the following equality.

\[\max \left\{\sum_{i\in [N]}x_i : A\xxx\le \vone,\xxx\ge \vzero \right\}=
\max \left\{\sum_{i\in [N]}x_i : A\xxx\le \vone,\xxx\ge \vzero, \xxx \mbox{ is $\cO$-regular}\right\}.\]

In other words, we reduce the number of variables from $N$ to $n$.
We may then rewrite the linear program as follows:
\[ \max \left\{  \sum_{i\in [n]} |O_i|x^*_i : A^*\xxx^*\le \vone, \xxx^*\ge \vzero   \right\}, \]
where $A^*$ is an $M\times n$ matrix defined by
\[A^*(i,k)= \sum_{j\in O_k} A(i,j) \mbox{ for all $i\in[M], \, k\in[n]$}.\]

\subsection{Reducing the dimension of the linear program defined by \eqref{comp:lp}}

To exploit the symmetries of our linear program, we consider a domination graph with an equitable degree distribution.
Specifically, for given values of $m$ and $n$, we set $\delta=\trunc{n/m}$, and
\begin{align*}
m_1&= n\bmod{m}, & m_2 &= m-m_1, \\ 
n_1&= m_1(\delta+1), & n_2 &= m_2\delta.  
\end{align*}
We partition $[m]$ into $I_1\deff [m_1]$ and $I_2\deff [m_1+1,m]$ 
(here, $[i,j]$ denotes the set $\{i,i+1,\ldots, j\}$)
and assign the vertices in $I_1$ and $I_2$ the degrees $\delta+1$ and $\delta$, respectively.
Hence, in this domination graph, there are $m_1$ vertices of degree $\delta+1$ and $m_2$ vertices of degree $\delta$.

We also partition $[n]$ into $m$ groups:
\[ J_i\deff \begin{cases}
[(i-1)(\delta+1)+1,i(\delta+1)], & \mbox{if $i\in [m_1]$},\\
[n_1+(i-m_1-1)\delta+1,n_1+(i-m_1)\delta], & \mbox{otherwise}.
\end{cases}\]
In other words, we partition $[n_1]$ into $m_1$ groups of size $\delta+1$ and 
$[n_1+1,n]$ into $m_2$ groups of size $\delta$.

Next, we define a group of $\bA$-preserving permutations. 
Here, let $\bbS_X$ denote the set of permutations on the set $X$
and we first produce permutations in $\bbS_{\cS}$ and $\bbS_{\cR}$.
Consider a permutation $\gamma\in \bbS_{I_1}\times \bbS_{I_2}$.
Since $I_1$ and $I_2$ partition $[m]$, the permutation $\gamma$ belongs to $\bbS_{[m]}$
and we define $\pi^{\cS}_\gamma\in \bbS_{\cS}$ such that $\pi^{\cS}_\gamma(\uuu)=\uuu^\gamma$ for $\uuu\in \cS$.
We then consider the subset
\[\bbB=\{\beta\in\bbS_{\cR}: \beta(\vvv)|_{J_i}\ne \vzero \mbox{ if and only if } \vvv|_{J_i}\ne \vzero \mbox{ for all } i\in [m],\, \vvv\in \cS\}.\] 
It can be verified that $\bbB$ is a subgroup of $\bbS_{\cR}$.
Given $\gamma\in \bbS_{I_1}\times \bbS_{I_2}$ and $\beta\in\bbB$, we define 
$\pi^{{\cR}}_{\gamma,\beta}\in \bbS_{\cR}$ to be the permutation such that $\pi^{{\cR}}_{\gamma,\beta}(\vvv)$ 
is the word obtained rearranging the $m$ subwords in $\beta(\vvv)$ in accordance to $\gamma$.
Finally, we obtain a permutation $\pi_{\gamma,\beta}$ on ${\cS}\times {\cR}$ by simply setting 
$\pi_{\gamma,\beta}(\uuu,\vvv)=\left(\pi^{{\cS}}_\gamma(\uuu), \pi^{{\cR}}_{\gamma,\beta}(\vvv)\right)$.
Set $\bbG=\{\pi_{\gamma,\beta}: \gamma\in \bbS_{I_1}\times \bbS_{I_2}, \beta\in \bbB\}$.
Abusing notation, we simply write $\bbG=(\bbS_{I_1}\times \bbS_{I_2})\times \bbB$.

\begin{example}\label{exa:A1}\hfill
\begin{enumerate}[(i)]
\item (Example \ref{exa:A} continued.) Consider $m=2$, $n=4$, and $w=1$. 
Then $\delta=2$, $m_1=n_1=0$, $m_2=2$ and $n_2=4$.

Label the words in $\cR$ as follow: $\vvv_0=0000$, $\vvv_1=0001$, $\vvv_2=0010$, $\vvv_3=0100$, and
$\vvv_4=1000$. So, $\bbB=\{(\vvv_0,\vvv_1,\vvv_2,\vvv_3,\vvv_4),(\vvv_0,\vvv_2,\vvv_1,\vvv_3,\vvv_4),(\vvv_0,\vvv_1,\vvv_2,\vvv_4,\vvv_3),(\vvv_0,\vvv_2,\vvv_1,\vvv_4,\vvv_3)\}$.
Since $m_1=0$, $\bbG=\bbS_{[2]}\times \bbB$.

Let $\gamma=(2,1)$ and $\beta=(\vvv_0,\vvv_2,\vvv_1,\vvv_3,\vvv_4)$. 
Then $\pi^{\cS}_\gamma(01)=10$, $\pi^{\cR}_{\gamma,\beta}(0001)=1000$ and 
so, $\pi_{\gamma,\beta}(01,0001)=(10,1000)$.
Consider the orbit of $(01,0001)$ which is $\{\pi(01,0001): \pi\in \bbG\}$. 
The orbit is given by \[\{(01,0001), (01,0010),(10,0100),(10,1000)\}.\]

\item Consider $m=3$, $n=4$, and $w=2$. 
Then $\delta=1$, $m_1=1$, $m_2=2$, $n_1=2$ and $n_2=2$.

The orbit of $(101,0101)$ is given by $\{(101,0101), (101,1001),(110,0110),(110,1010)\}$,
while the orbit of $(100,0100)$ is given by $\{(100,0100), (100,1000)\}$. 

\end{enumerate}
\end{example}

Recall that the columns of $\bA$ are indexed by $E\subseteq \cS \times \cR$,
while $\bbG$ is a subgroup of $\bbS_{\cS\times \cR}$.
In the next lemma, we show that $\bbG$ may be regarded as a subgroup of $\bbS_E$.
We do this by showing that the image of $E$ under the permutation $\pi_{\gamma,\beta}$ remains as $E$.

\begin{lemma}
Let $\pi_{\gamma,\beta}\in \bbG$, $\uuu\in \cS$ and $\vvv\in \cR$. Then $\uuu$ dominates $\vvv$
if and only if $\pi^{\cS}_{\gamma}(\uuu)$ dominates $\pi^{\cR}_{\gamma,\beta}(\vvv)$. 
\end{lemma}

\begin{proof}
Since $\bbG$ is a group, it suffices to prove in one direction.
Let $\supp_{\delta}(\vvv)\deff \{i\in [m]: \vvv|_{J_i}\ne \vzero\}$. 
Since $\uuu$ dominates $\vvv$, we have that $\supp(\uuu)\supseteq \supp_\delta(\vvv)$.

Since $\beta\in \bbB$, we have that 
$\supp_\delta(\vvv)=\supp_\delta(\beta(\vvv))$. 
From the definition of $\pi^{\cS}_{\gamma}$ and $\pi^{\cR}_{\gamma,\beta}$,
we have that $\supp(\pi^{\cS}_{\gamma}(\uuu))=\gamma(\supp(\uuu))$ and 
$\supp_\delta(\pi^{\cR}_{\gamma,\beta}(\vvv))=\gamma(\supp_\delta(\beta(\vvv)))=\gamma(\supp_\delta(\vvv))$.
Therefore, $\supp(\pi^{\cS}_{\gamma}(\uuu))\supseteq \supp_\delta(\pi^{\cR}_{\gamma,\beta}(\vvv))$ and 
so, $\pi^{\cS}_{\gamma}(\uuu)$ dominates $\pi^{\cR}_{\gamma,\beta}(\vvv)$.
\end{proof}

Therefore, since $E=\{(\uuu,\vvv)\in {\cS}\times {\cR}: \mbox{$\uuu$ dominates $\vvv$}\}$, 
we have that $\bbG$ is a subgroup of $\bbS_E$.

\begin{lemma}
Every permutation in $\bbG$ is a $\bA$-preserving permutation.
\end{lemma}

\begin{proof}
Let $\pi=\pi_{\gamma,\beta}\in \bbG$. 
To show that $\pi$ is $\bA$-preserving, it suffices to provide a permutation $\pi_{\rm row}$ on the rows of $\bA$
such that $P_{\pi_{\rm row}}\bA P_{\pi}=\bA$, or 
\begin{equation}\label{eq:rowandcol}
\bA P_{\pi}=P_{\pi_{\rm row}}^{-1}\bA=P_{\pi_{\rm row}^{-1}}\bA.
\end{equation}

Let $(\uuu,\vvv)\in E$ and $\zzz\in {\cS}\cup {\cR}$. 
Let the entry of $\bA$ at row $\zzz$ and column $(\uuu,\vvv)$, 
or the $(\zzz,(\uuu,\vvv))$th entry of $\bA$, be written as $\bA(\zzz,(\uuu,\vvv))$.
Then the $(\zzz,(\uuu,\vvv))$th entry of $\bA P_{\pi}$ is given by 
$\bA(\zzz,\pi^{-1}(\uuu,\vvv))$.

We next consider the permutations on the rows. 
Note that $\pi^{\cS}_\gamma \in\bbS_{{\cS}}$ and $\pi^{\cR}_{\gamma,\beta}\in \bbS_{\cR}$.
So, let $\bbG$ act on ${\cS}\cup {\cR}$ by setting $\pi(\zzz)=\pi^{\cS}_\gamma(\zzz)$ when $\zzz\in {\cS}$ and
$\pi(\zzz)=\pi^{\cR}_{\gamma,\beta}(\zzz)$ when $\zzz\in {\cR}$.
Set $\pi_{\rm row}=\pi^{-1}$, the inverse of $\pi$ in $\bbG$.
Then the $(\zzz,(\uuu,\vvv))$th entry of $P_{\pi_{\rm row}^{-1}}\bA$ is given by 
$\bA(\pi_{\rm row}^{-1}(\zzz),(\uuu,\vvv))=\bA(\pi(\zzz),(\uuu,\vvv))$.

Finally, to establish \eqref{eq:rowandcol} we show that 
$A(\zzz,\pi^{-1}(\uuu,\vvv))=A(\pi(\zzz),(\uuu,\vvv))$. 
This then follows from the fact that
\[
A(\zzz,(\uuu,\vvv))=\begin{cases}
1, &\mbox{if $\zzz=\uuu$ or $\zzz=\vvv$},\\
0, &\mbox{otherwise}. \end{cases}
\] 
\end{proof}

We next study the orbits of ${\cS}\times {\cR}$ under this group action of $\bbG$.
To this end, we define the $(m_1,m_2)$-weights and the $(n_1,n_2;\delta)$-weights of words in 
${\cS}$ and ${\cR}$, respectively. 
For $\uuu\in {\cS}$ and $\vvv\in {\cR}$, define 
\begin{align*}
\wts(\uuu) &=(\sigma_1,\sigma_2), \mbox{ where } \sigma_i=\wt(\uuu|_{I_i})\mbox{ for }i\in[2],\\
\wtr(\vvv) &=(\rho_1,\rho_2), \mbox{ where } \rho_i=|\{j\in I_j: \vvv|_{J_j}\ne \vzero\}|\mbox{ for }i\in[2].
\end{align*}

Using these weights, we then characterise the orbits of ${\cS}\times {\cR}$ under this group action of $\bbG$.

\begin{lemma}\label{lem:wt-orbit}
Let $(\uuu,\vvv), (\uuu',\vvv')\in E$. 
$(\uuu,\vvv)$ and $(\uuu',\vvv')$ belong to the same orbit if and only if 
$\wts(\uuu)=\wts(\uuu')$ and $\wtr(\vvv)=\wtr(\vvv')$.
\end{lemma}

\begin{proof}
If $(\uuu,\vvv)$ and $(\uuu',\vvv')$ belong to the same orbit, then $\pi_{\gamma,\beta}(\uuu,\vvv)=(\uuu',\vvv')$
for some $\pi_{\gamma,\beta}\in\bbG$. 
In other words, $\pi^{\cS}_{\gamma}(\uuu)=\uuu'$ and $\pi^{\cR}_{\gamma,\beta}(\vvv)=\vvv'$. 
Since $\gamma\in \bbS_{I_1}\times \bbS_{I_2}$, we have $\wt(\uuu|_{I_i})=\wt(\uuu'|_{I_i})$ for $i\in [2]$, and so, $\wts(\uuu)=\wts(\uuu')$.
Similarly, since $\beta\in \bbB$, we have $|\{j\in I_j: \vvv|_{J_j}\ne \vzero\}|=|\{j\in I_j: \vvv'|_{J_j}\ne \vzero\}|$ and 
hence, $\wtr(\vvv)=\wtr(\vvv')$.

Conversely, suppose that $\wts(\uuu)=\wts(\uuu')=(\sigma_1,\sigma_2)$ and $\wtr(\vvv)=\wtr(\vvv')=(\rho_1,\rho_2)$.
Consider the words 
\begin{align*}
\uuu^*&=\overbrace{00\cdots 0\underbrace{11\cdots 1}_{\sigma_1}}^{m_1}\overbrace{00\cdots 0\underbrace{11\cdots 1}_{\sigma_2}}^{m_2}\, ,\\
\vvv^*&=
\underbrace{\overbrace{0\cdots 0}^{\delta+1}\cdots\overbrace{0\cdots 0}^{\delta+1}}_{m_1-\rho_1}
\underbrace{\overbrace{0\cdots 01}^{\delta+1}\cdots\overbrace{0\cdots 01}^{\delta+1}}_{\rho_1}
\underbrace{\overbrace{0\cdots 0}^{\delta}\cdots\overbrace{0\cdots 0}^{\delta}}_{m_2-\rho_2}
\underbrace{\overbrace{0\cdots 01}^{\delta}\cdots\overbrace{0\cdots 01}^{\delta}}_{\rho_2}.
\end{align*}
Then $\wts(\uuu^*)=(\sigma_1,\sigma_2)$ and $\wtr(\vvv^*)=(\rho_1,\rho_2)$.
We find a permutation in $\bbG$ that maps $(\uuu^*,\vvv^*)$ to $(\uuu,\vvv)$.
Let $\gamma^{-1}$ be a permutation in $\gamma\in \bbS_{I_1}\times \bbS_{I_1}$
that rearranges the coordinates the $m$ coordinates of $\uuu$ and the $m$ subwords of $\vvv$
such that $(\uuu^{\gamma^{-1}},\hat{\vvv})$ is the lexicographical smallest amongst all permutations.
Let $\beta\in \bbB$ be a permutation that maps $\vvv^*$ to $\hat{\vvv}$.
Then $\pi_{\gamma,\beta}(\uuu^*,\vvv^*)=(\uuu,\vvv)$.
Since $\bbG$ is a subgroup of $\bbS_{E}$, 
we can then find a permutation that maps $(\uuu,\vvv)$ to $(\uuu',\vvv')$.
\end{proof}

Following Lemma \ref{lem:wt-orbit}, we can then index each orbit with the quadruple 
$(\sigma_1,\sigma_2,\rho_1,\rho_2)$, where
 $(\sigma_1,\sigma_2)=\wts(\uuu)$ and $(\rho_1,\rho_2)=\wtr(\vvv)$
 for some $(\uuu,\vvv)$ in the orbit. In particular, the index set is given by
 \[\Omega\deff \{(\sigma_1,\sigma_2,\rho_1,\rho_2): 
 \mbox{ $0\le \sigma_i\le m_i$ for $i\in[2]$, $0\le \rho_1\le \min\{\sigma_1,w\}$, 
 $0\le \rho_2\le \min\{\sigma_2,w-\rho_1\}$}\}.\]
 
Hence, the number of variables is reduced to $O(m^2w^2)$.
Besides reducing the number of variables, the group action also identifies certain constraints.
In particular, the $2^m+\sum_{j=0}^w \binom{n}{j}$ constraints are reduced to  $O(m^2+w^2)$ constraints.
Then by the equivalence of the linear programs, we are able to determine the existence of the 
desired mapping in time polynomial in $m$ and $w$.

Next, we compute the size of the orbits and then state the reduced linear program.
Let $0\le \rho_1,\rho_2\le w$ and $0\le \rho_1+\rho_2\le w$ and define the quantity 
\[C_{\rho_1\rho_2}\deff 
\sum_{\substack{w^{(1)}_{1}+\cdots+w^{(1)}_{\rho_1}+w^{(2)}_{1}+\cdots+w^{(2)}_{\rho_2}\le w\\w^{(i)}_j\ge 1}}
\prod_{j=1}^{\rho_1}\binom{\delta+1}{w^{(1)}_j}\prod_{j=1}^{\rho_2}\binom{\delta}{w^{(2)}_j}.\]
Here, $C_{\rho_1\rho_2}$ computes the number of words $\vvv$ in ${\cR}$ such that
$\{j\in I_1: \vvv|_{J_j}\ne \vzero\}=L_1$ and $\{j\in I_2: \vvv|_{J_j}\ne \vzero\}=L_2$
for some $\rho_1$-subset $L_1$ of $I_1$ and $\rho_2$-subset $L_2$ of $I_2$.
 
\begin{lemma}\label{lem:orbitsize}
Fix $(\sigma_1,\sigma_2,\rho_1,\rho_2)\in\Omega$.
\begin{enumerate}[(i)]
\item The number of pairs $(\uuu,\vvv)$ in $E$ with $\wts(\uuu)=(\sigma_1,\sigma_2)$ and  
$\wtr(\vvv)=(\rho_1,\rho_2)$
is $\binom{m_1}{\sigma_1}\binom{m_2}{\sigma_2}\binom{\sigma_1}{\rho_1}\binom{\sigma_2}{\rho_2}C_{\rho_1\rho_2}$.
\item Fix $\uuu\in {\cS}$ with $\wts(\uuu)=(\sigma_1,\sigma_2)$. 
The number of words $\vvv\in {\cR}$ that are dominated by $\uuu$ with $\wtr(\vvv)=(\rho_1,\rho_2)$
is $\binom{\sigma_1}{\rho_1}\binom{\sigma_2}{\rho_2}C_{\rho_1\rho_2}$.
\item Fix $\vvv\in {\cR}$ with $\wtr(\vvv)=(\sigma_1,\sigma_2)$. 
The number of words $\uuu\in {\cS}$ that dominates $\vvv$ with $\wts(\uuu)=(\sigma_1,\sigma_2)$
is $\binom{m_1-\rho_1}{\sigma_1-\rho_1}\binom{m_2-\sigma_2}{\sigma_2-\rho_2}$.
\end{enumerate}
\end{lemma}

\begin{proof}
For $(\uuu,\vvv)\in E$, we set $K_i$ to be the support of $\uuu|_{I_i}$ and 
$L_i=\{j\in I_i: \vvv|_{J_j}\ne \vzero\}$ for $i\in [2]$.
Note that if $\wts(\uuu)=(\sigma_1,\sigma_2)$ and $\wtr(\vvv)=(\rho_1,\rho_2)$,
then $|K_i|=\sigma_i$ and $|L_i|=\rho_i$ for $i\in[2]$.
\begin{enumerate}[(i)]
\item There are $\binom{m_1}{\sigma_1}\binom{m_2}{\sigma_2}$ to choose words $\uuu$ with
$\wts(\uuu)=(\sigma_1,\sigma_2)$. 
Since $\uuu$ dominates $\vvv$, we have $\binom{\sigma_1}{\rho_1}\binom{\sigma_2}{\rho_2}$ choices  
for $L_1$ and $L_2$. For fixed $L_1$ and $L_2$, there are $C_{\rho_1\rho_2}$ choices for $\vvv$.
Therefore, the total number of pairs is $\binom{s_1}{\sigma_1}\binom{s_2}{\sigma_2}\binom{\sigma_1}{\rho_1}\binom{\sigma_2}{\rho_2}C_{\rho_1\rho_2}$.
\item When $K_1$ and $K_2$ are fixed, the previous argument demonstrates that there are 
$\binom{\sigma_1}{\rho_1}\binom{\sigma_2}{\rho_2}C_{\rho_1\rho_2}$ choices for $\vvv$.
\item When $L_1$ and $L_2$ are fixed, there are $\binom{m_1-\rho_1}{\sigma_1-\rho_1}$ and $\binom{m_2-\rho_2}{\sigma_2-\rho_2}$ choices for $K_1$ and $K_2$, respectively. 
Therefore, the desired number is $\binom{s_1-\rho_1}{\sigma_1-\rho_1}\binom{s_2-\sigma_2}{\sigma_2-\rho_2}$.
\end{enumerate}
\end{proof}
 
Finally, we state the reduced linear program.

\begin{equation}\label{comp:reduced-lp}
 \max \sum_{(\sigma_1,\sigma_2,\rho_1,\rho_2)\in\Omega}  \binom{m_1}{\sigma_1}\binom{m_2}{\sigma_2}\binom{\sigma_1}{\rho_1}\binom{\sigma_2}{\rho_2}C_{\rho_1\rho_2}X_{\sigma_1\sigma_2\rho_1\rho_2}
\end{equation}
subject to the following constraints.
\begin{enumerate}[(I)]
\item ${\cS}$-side constraints:
\[\sum_{\rho_1=0}^{\min\{\sigma_1,w\}}\sum_{\rho_2=0}^{\min\{\sigma_2,w-\rho_1\}} 
\binom{\sigma_1}{\rho_1}\binom{\sigma_2}{\rho_2}C_{\rho_1\rho_2} X_{\sigma_1\sigma_2\rho_1\rho_2}\le 1 \mbox{ for all } 0\le \sigma_1\le m_1,\, 0\le\sigma_2\le m_2.\]
\item ${\cR}$-side constraints:
\[\sum_{\sigma_1=\rho_1}^{m_1}\sum_{\sigma_2=\rho_2}^{m_2} 
\binom{m_1-\rho_1}{\sigma_1-\rho_1}\binom{m_2-\rho_2}{\sigma_2-\rho_2} X_{\sigma_1\sigma_2\rho_1\rho_2}\le 1 \mbox{ for all } 0\le \rho_1\le m_1,\, 0\le \rho_2\le m_2,\, 0\le \rho_1+\rho_2 \le w.\]
\item Variable constraints:
\[ 0\le X_{\sigma_1\sigma_2\rho_1\rho_2}\le 1 \mbox{ for all } (\sigma_1,\sigma_2,\rho_1,\rho_2)\in\Omega. \]
\end{enumerate}

There exists a mapping if and only if the objective value achieves $2^m$. Furthermore, we know that the objective value attains $2^m$ if and only if the $\cS$-side constraints are active, or met with equality.

\begin{example}\label{exa:A2}\hfill
\begin{enumerate}[(i)]
\item (Example \ref{exa:A} continued.) Consider $m=2$, $n=4$, and $w=1$. 
Then $\delta=2$, $m_1=n_1=0$, $m_2=2$, $n_2=4$.

Hence, $E$ is partitioned into the orbits
\[ \cO=\{ O_{(0,0)}=\{e_1\},  O_{(1,0)}=\{e_2,e_5\},  O_{(1,1)}=\{e_3,e_4,e_6,e_7\},  
O_{(2,0)}=\{e_8\}, 
O_{(2,1)}=\{e_9,e_{10},e_{11},e_{12}\} \},\]
with
\[\Omega=\{(0,0),(1,0),(1,1),(2,0),(2,1)\}.\]

Then $A^*$ is given by the $5\times 5$ matrix
\[ A^*=\left(
\begin{array}{c cc cc}
1& 0&0& 0&0\\
0& 1&2& 0&0\\
0& 0&0& 1&4\\
1& 2&0& 1&0\\
0& 0&1& 0&1
\end{array}
\right)\, .
\]

\item Consider $m=3$, $n=4$, and $w=2$. 
Then $\delta=1$, $m_1=1$, $m_2=2$, $n_1=2$ and $n_2=2$.
Here,
{\small
\begin{align*}
\Omega &=\{(0,0,0,0),(0,1,0,1),(0,1,0,1),(0,2,0,0),(0,2,0,1),(0,2,0,2),(1,0,0,0),(1,0,1,0),\\
&\hspace{7mm}(1,1,0,0),(1,1,0,1),(1,1,1,0),(1,1,1,1),(1,2,0,0),(1,2,0,1),(1,2,0,2),(1,2,1,0),(1,2,1,1)\}.
\end{align*}
}

Then $A^*$ is given by the $11\times 17$ matrix
\[ A^*=\left(
\begin{array}{c cc ccc cc cccc ccccc}
1& 0&0& 0&0&0& 0&0& 0&0&0&0& 0&0&0&0&0\\
0& 1&1& 0&0&0& 0&0& 0&0&0&0& 0&0&0&0&0\\
0& 0&0& 1&2&1& 0&0& 0&0&0&0& 0&0&0&0&0\\
0& 0&0& 0&0&0& 1&3& 0&0&0&0& 0&0&0&0&0\\
0& 0&0& 0&0&0& 0&0& 1&1&3&2& 0&0&0&0&0\\
0& 0&0& 0&0&0& 0&0& 0&0&0&0& 1&2&1&3&4\\

1& 2&0& 1&0&0& 1&0& 2&0&0&0& 1&0&0&0&0\\
0& 0&1& 0&1&0& 0&0& 0&1&0&0& 0&1&0&0&0\\
0& 0&0& 0&0&1& 0&0& 0&0&0&0& 0&0&1&0&0\\
0& 0&0& 0&0&0& 0&1& 0&0&2&0& 0&0&0&1&0\\
0& 0&0& 0&0&0& 0&0& 0&0&0&1& 0&0&0&0&1
\end{array}
\right)\, .
\]
\end{enumerate}
\end{example}

We summarize the results in this section with the following theorem.

\begin{theorem}\label{theorem-poly-time}
Let $G=([m]\cup[n], E)$ be an equitable domination graph.
The linear program defined by \eqref{comp:reduced-lp} is equivalent to \eqref{comp:lp} and
has $O(m^2w^2)$ variables and $O(m^2+w^2)$ constraints.
Therefore, we can determine the existence of an $G$-domination mapping in time polynomial in $m$ and $w$.
\end{theorem}

\section{Existence Proof}
\label{sec:proof}

%

In this section, we prove Theorem~\ref{theorem-main}. 
In particular, we show that for sufficiently large $m$,
an $(m,n,w)$-domination mapping exists whenever $|\cB(n,w)|\ge 2^m$.
From Theorem~\ref{thm:sub_dominate}, this is equivalent to finding a perfect matching in 
the associated compatibility graph.
To this end, we invoke the celebrated Hall's marriage theorem~\cite{Hall35}.

\begin{theorem}
\label{thm:Hall}
Consider a bipartite graph with $V_1$ and $V_2$ as its left and right vertices, 
and $|V_1| \leq |V_2|$.
There is a matching of size $|V_1|$ (i.e. a perfect matching) if and only if for each subset $X$ of $V_1$, the number
of vertices in the neighborhood of $X$ has at least $|X|$ vertices.
\end{theorem}

In what follows, we study the suffciency condition of Theorem~\ref{thm:Hall} for our compatibility graph.
In particular, a set of vertices $X\subset \cS$  is called a {\em bad set}
if its neighbourhood $N(X)$ in $\cR$ has size less than $|X|$. In other words, $|X|>|N(X)|$.
We examine properties of bad sets and find compatibility graphs that contain no bad sets. 
Therefore, in such compatibility graphs, the condition of Theorem~\ref{thm:Hall} is satisfied and 
hence the corresponding $(m,n,w)$-domination mapping exists.


For our exposition, we define certain quantities and properties for sets of vertices in $\cS$.
First, we define the notion of descendant-closed.

\begin{definition} Let $\uuu,\vvv\in \cS$. Recall that $\vvv\prec \uuu$ if $\supp(\vvv)\subseteq \supp(\uuu)$ and 
call $\vvv$ a {descendant} of $\uuu$. 
A set $X$ is {\bfit descendant-closed} or {\bfit d-closed} if for all $\uuu\in X$, $\vvv\prec \uuu$ implies that  $\vvv\in X$.
\end{definition}

Given any bad set, we are always able to construct a (possibly different) d-closed bad set with the same cardinality. 

\begin{lemma}[Closure Lemma] 
If $X$ is a bad set, then there exists a d-closed bad set $X'$ for which $|X|=|X'|$.
\end{lemma}

\begin{proof}
Suppose $X$ is not d-closed. Then there exists $\uuu\in X$ and $\vvv\in {\cal S}$ such that $\vvv\prec\uuu$ but $\vvv\notin X$.
Set $X^{(1)}= X\cup \{\vvv\}\setminus \{\uuu\}$. 
Then $|X^{(1)}|=|X|$ and $|N(X^{(1)})|\le |N(X)|<|X|=|X^{(1)}|$.
Hence, $X^{(1)}$ is a bad set with the same size as $X$. 
If $X^{(1)}$ is not d-closed, then we similarly construct $X^{(2)}$ such that $|X^{(2)}|=|X|$ and $X^{(2)}$ is bad.
Since this process is finite, we eventually obtain a d-closed set $X^{(m)}$ such that $|X^{(m)}|=|X|$ and $X^{(m)}$ is bad.
\end{proof}

Next, we define the notion of balanced sets.

\begin{definition} Let $i\in [m]$. We say that $X$ is {\bfit $i$-balanced} if the following holds.
If $\aaa\bbb\in X$ with $|\aaa|=i$, then $\aaa'\bbb\in X$ for all $\aaa'\in \{0,1\}^i$.
\end{definition}

Observe that if $X$ is $i$-balanced, then $X$ is $j$-balanced for all $j\le i$.
As a convention, we say that all sets $X$ are $0$-balanced.
Finally, we look at certain collection of bad sets that are both d-closed and balanced.

\begin{definition} Let $0\le i\le s$. Let $\cX_i$ be the collection of all bad sets
that are d-closed and $i$-balanced.
\end{definition}

Our proof on the non-existence of bad sets follows an induction strategy.

\begin{enumerate}[(I)]
\item Assume that there exists a bad set $X_0$.
\item {\bf Base case}. Using the closure lemma, we may modify $X_0$ so that $X_0'$ is both d-closed and 0-balanced. 
This implies that $\cX_0$ is nonempty.
\item {\bf Induction step}. Let $1\le i\le m$. By induction hypothesis, we have that $\cX_{i-1}$ is nonempty.
Pick $X\in \cX_{i-1}$ with the smallest cardinality. We then construct a set $X'$ such that $X'\in \cX_{i}$.
This implies that $\cX_{i}$ is nonempty.
\item Therefore, $\cX_m$ is nonempty and $X\in \cX_m$ implies that $X=\cS$. 
This contradicts the fact that $|\cS|\le |\cR|=|N(\cS)|$.
\end{enumerate}

In what follows, we focus on the induction step. 
We fix $X\in \cX_{i-1}$ and augment $X$ so that the resulting $X'$ is in $\cX_i$. 
To ensure that $X'$ is bad, we need to compute the increase in the size of the neighbourhood.
Formally, we have the following definition.

\begin{definition} Let $U, V\subseteq \cS$. Set $\Xi_U(V)\deff |N(U\cup V)|-|N(U)|$.
In other words, $\Xi_U(V)$ is the number of additional neighbours that $V$ adds when adjoined to $U$.
\end{definition}

When $X$ is the smallest set in $\cX_{i-1}$, we have the following property of $\Xi_{X\setminus Y}(Y)$ for certain subset $Y\subseteq X$.

\begin{lemma}[Removal Lemma]
Let $X$ be the smallest set in $\cX_i$.
Let $Y\subseteq X$ such that $X\setminus Y$ is both d-closed and $i$-balanced.
Then $|Y|>\Xi_{X\setminus Y}(Y)$.
\end{lemma}

\begin{proof}
Since $X$ is a bad set, ww have that  $|X|>|N(X)|$.
By the minimality of $X$, we have that $X\setminus Y$ is not a bad set, because $X\setminus Y$ is both d-closed and $i$-balanced and $|X\setminus Y|<|X|$.
Therefore, $|X\setminus Y|\le |N(X\setminus Y)|$. Hence,
$|Y|=|X|-|X\setminus Y|> |N(X)|- |N(X\setminus Y)|=\Xi_{X\setminus Y}(Y)$.
\end{proof}


Next, we describe how we augment $X$. Define
\begin{align*}
Y & =\{\uuu\in X : u_{i}=0, \, \uuu+\eee_{i}\notin X\}, \mbox{ and }\\
Z &=Y+\eee_{i}.
\end{align*}

We have the following lemmas on $X\cup Z$ and $X\setminus Y$.

\begin{lemma}\label{lem:XcupZ}
If $X$ is d-closed and $(i-1)$-balanced, then $X\cup Z$ is both d-closed and $i$-balanced.
\end{lemma}

\begin{proof}
First, we show that $X\cup Z$ is d-closed.
Let $\uuu\in X\cup Z$ and $\vvv\prec \uuu$. 
If $\uuu\in X$, then $\vvv\in X\subseteq X\cup Z$ since $X$ is d-closed.
Otherwise, $\uuu\in Z$ and hence, $\uuu=\uuu'+\eee_i$ with $\uuu'\in X$.
If $v_i=0$, then $\vvv\prec \uuu'$ and so, $\vvv\in X$ since $X$ is d-closed.
Hence, $\vvv\in X\cup Z$.
If $v_i=1$, then $\vvv-\eee_i \prec \uuu'$ and so, $\vvv-\eee_i\in X$.
Hence, $\vvv=(\vvv-\eee_i)+\eee_i\in Z\subseteq X\cup Z$.
Thus, $X\cup Z$ is d-closed.

Next, we show that $X\cup Z$ is $i$-balanced. 
Let $\aaa\bbb\in X\cup Z$ with $|\aaa|=i$ and let $\ccc\in \{0,1\}^i$.
Let $\aaa'$ and $\ccc'$ be the first $i-1$ symbols of $\aaa$ and $\ccc$, respectively.
So, $\aaa=\aaa'a_i$ and $\ccc=\ccc'c_i$.
Since $X\cup Z$ is d-closed, it suffices to consider the case, $c_i=1$.

In other words, we want to show that $\ccc'1\bbb \in X\cup Z$.
We distinguish between two cases:
\begin{enumerate}[(a)]
\item $\aaa'1\bbb\in X$. 
Since $X$ is $(i-1)$-balanced, we have  $\ccc'1\bbb\in X\subseteq X\cup Z$.
\item $\aaa'1\bbb\notin X$. This implies that $\aaa'0\bbb\in X$ by the definition of $Y$ and $Z$. 
Again, since $X$ is $(i-1)$-balanced,
we have $\ccc'0\bbb\in X$ and so, $\ccc'1\bbb=\ccc'0\bbb+\eee_i\in X\cup Z$.
\end{enumerate}
Thus, $X\cup Z$ is $i$-balanced. 
\end{proof}

\begin{lemma}\label{lem:XminusY}
If $X$ is d-closed and $(i-1)$-balanced,  then $X\setminus Y$ is both d-closed and $(i-1)$-balanced.
\end{lemma}

\begin{proof}First, we show that $X\setminus Y$ is d-closed.
Let $\uuu\in X\setminus Y$ and $\vvv\prec \uuu$.
Since $\uuu\notin Y$, either $u_i=1$ or $\uuu+\eee_i\in X$.
Hence, set $\uuu'=\uuu$ if $u_i=1$ and $\uuu'=\uuu+\eee_i$ otherwise.
If $v_i=0$, then $\vvv+\eee_i \prec \uuu'\in X$, and so, $\vvv+\eee_i\in X$. 
Hence, $\vvv\notin Y$. If $v_i=1$, then $\vvv\notin Y$ by definition.
In both cases, $\vvv\notin Y$. Since $X$ is d-closed, $\vvv\in X$ and therefore, $\vvv\in X\setminus Y$.
Thus, $X\setminus Y$ is d-closed.

Next, we show that $X\setminus Y$ is $(i-1)$-balanced.
Let $\aaa\bbb\in X\setminus Y$ with $|\aaa|=i-1$ and let $\ccc\in \{0,1\}^{i-1}$.
Since $X$ is $(i-1)$-balanced, we have that $\ccc\bbb\in X$.
When $b_1=0$, we have by definition of $Y$ that $\aaa\bbb+\eee_i\in X$ since $\aaa\bbb\notin Y$.
Since $X$ is $(i-1)$-balanced and  $\aaa\bbb+\eee_i\in X$, 
we have that $\ccc\bbb+\eee_i \in X$ and so, $\ccc\bbb\notin Y$.
When $b_1=1$, we have that $\ccc\bbb\notin Y$ by definition.
Therefore, since $\ccc\bbb\in X$, it follows that $\ccc\bbb\in X\setminus Y$.
Thus, $X\setminus Y$ is $(i-1)$-balanced.
\end{proof}

Finally, we have the following lemma that reduces the induction step to 
demonstrating inequality \eqref{eq:target1}.

\begin{lemma} \label{lem:target1} If $X$ is an d-closed, $(i-1)$-balanced and bad set, and 
\begin{equation}\label{eq:target1}
\Xi_X(Z)\le \Xi_{X\setminus Y}(Y),
\end{equation}
then $X\cup Z$ is a d-closed, $i$-balanced and bad set. In other words, $X\cup Z\in \cX_i$.
\end{lemma}

\begin{proof}
By Lemma \ref{lem:XcupZ}, it remains to show that $X\cup Z$ is a bad set.

Since $X\setminus Y$ is both d-closed and $(i-1)$-balanced by Lemma \ref{lem:XminusY},
we apply the removal lemma to have that $|Y|>\Xi_{X\setminus Y}(Y)$.
Futhermore, we have that $|X|+|Z|=|X\cup Z|$ and $|Z|=|Y|$.
Combining these inequalities and equalities, we have that 
\begin{align*}
|X\cup Z| & = |X|+|Z|\\
&=|X|+|Y|\\
&>|N(X)|+ \Xi_{X\setminus Y}(Y)\\
&\ge |N(X)|+ \Xi_{X}(Z)=|N(X\cup Z)|. 
\end{align*}
Thus, $X\cup Z$ is a bad set.
\end{proof}

In view of Lemma \ref{lem:target1}, our goal is to show that \eqref{eq:target1} holds. To this end, let us first introduce more notation.

\begin{definition} For any vector $\vvv\in S$, we let $\Psi(\vvv)$ denote the {\bfit additional neighbourhood} of $\vvv$, i.e.  
\[\Psi(\vvv)\deff \big|\{\uuu\in N(\vvv): \uuu \notin N(\vvv') \mbox{ for all proper descendant $\vvv'\prec \vvv$}\}\big|.\]
For a set $V\subseteq S$, define $\Psi(V)=\sum_{v\in V}\Psi(v)$.
\end{definition}

The definition of $\Psi$ is useful for computing the size of certain neighbourhoods. 

\begin{lemma}\label{lem:Psi-closed}
For any d-closed set, we have that $|N(V)|=\Psi(V)$.
\end{lemma}

\begin{proof}
Suppose $V=\{\vvv_1,\vvv_2,\ldots, \vvv_m\}$, where $\wt(\vvv_j)\le \wt(\vvv_{j+1})$ for $j=1,2,\ldots, m-1$.
That is, we have ordered the vectors in $V$ by weight.
Let $V_j\deff \{\vvv_1,\vvv_2,\ldots, \vvv_j\}$ for $j=1,2,\ldots, m$ with the convention that
$V_0=\varnothing$. Then
\[ |N(V)|=\sum_{j=1}^{m}\Xi_{V_{j-1}}(V_{j})=\sum_{j=1}^m\Psi(\vvv_j)=\Psi(V). \]
$\Xi_{V_{j-1}}(V_{j})=\Psi(\vvv_{j})$ follows from the fact that any proper descendant of $\vvv_{j}$
is contained in $V_{j-1}$.
\end{proof}

\begin{lemma}\label{lem:Psi-addVtoU}
Suppose a subset $U\cup V$ of $\cS$ is d-closed.
Then \[\Xi_U(V)=\Psi(V),\]
provided that the sets $U$ and $V$ are disjoint, and 
$U$ does not contain an ancestor of any vector in $V$.
Hence, if $U\cup V$ and $U$ are d-closed and
$U$ and $V$ are disjoint, then $\Xi_U(V)=\Psi(V)$.
\end{lemma}

\begin{proof}
Similar to Lemma \ref{lem:Psi-closed}, we adjoin the vectors in $V$ to $U$ in the order of their weight
and define $V_0=U$ by convention. Then 
\[ \Xi_U(V)=\sum_{j=1}^{m}\Xi_{V_{j-1}}(V_{j})=\sum_{j=1}^m\Psi(\vvv_j)=\Psi(V). \]
as before. To see that $\Xi_{V_{j-1}}(V_{j})=\Psi(\vvv_{j})$, we note the following.
First, when $\vvv_j$ is adjoined to $V_{j-1}$, all of its descendants of $\vvv_j$ are already in $V_{j-1}$
since $U\cup V$ is d-closed.
Moreover, $\vvv_j$ is not in $V_{j-1}$ and none of its ancestors are in $V_{j-1}$.
\end{proof}

Lemma~\ref{lem:Psi-addVtoU} then allows us to compute the $\Xi_{ X\setminus  Y}( Y)$ and $\Xi_ X(Z)$
by simply evaluating $\Psi$.

\begin{corollary}\label{cor:Xi-YZ}
\[\Xi_{ X\setminus  Y}( Y) = \Psi( Y),\mbox{  and }
\Xi_{ X}(Z) = \Psi(Z).
\]
\end{corollary}


Combining Lemma \ref{lem:target1} and Corollary \ref{cor:Xi-YZ}, to complete the induction step, it is sufficient to show that
\begin{equation}\label{eq:target2}
\Psi(Z)\le \Psi( Y).
\end{equation}

\subsection{Equitable domination graphs}

We complete the proof of Theorem~\ref{theorem-main}. 
To this end, we consider an equitable domination graph.
If we set $\delta=\trunc{n/m}$, then the left vertices in the domination graph have degrees $\delta$ and $\delta+1$.

We have the following characterization of $\Psi(\vvv)$ in terms of its weight.

\begin{lemma}\label{lem-Psi-equitable}
For $\vvv\in\{0,1\}^m$, let the weight of $\vvv$ be $\ell$.
Then for sufficiently large $m$ (which implies sufficiently large $\delta$),
there exist constants $c_1,c_2,\ldots,c_w$ such that
\[
\Psi(\vvv)=
\begin{cases}
c_\ell \delta^w+O(\delta^{w-1}), & \mbox{if $\ell\le w$},\\
0, &\mbox{otherwise.}
\end{cases}
\]
Furthermore, when $\wt(\vvv)=w$, we have that $\Psi(\vvv)=\delta^j(\delta+1)^{w-j}\ge\delta^w$ for some $0\le j\le w$ .
\end{lemma}

\begin{proof}
It suffices to consider the case where $\ell\le w$.
Let $j$ be the number of vertices in $\supp(\vvv)$ with degree $\delta$. Then the number of vertices in $\supp(\vvv)$ with degree $\delta+1$ is $\ell-j$ and we have that
\begin{align*}
\Psi(\vvv)
& = \sum_{\substack{i_1+i_2+\cdots +i_\ell\le w \\i_1,i_2,\ldots, i_\ell>0}} \binom{\delta}{i_1}\binom{\delta}{i_2}\cdots\binom{\delta}{i_j}\binom{\delta+1}{i_{j+1}}\binom{\delta+1}{i_{j+2}}\binom{\delta+1}{i_{\ell}}\\
& = \sum_{\substack{i_1+i_2+\cdots +i_\ell= w \\i_1,i_2,\ldots, i_\ell>0}}
\left(\lambda_{i_1,i_2,\ldots,i_\ell}\delta^w+O(\delta^{w-1}) \right) + O(\delta^{w-1})
= c_\ell \delta^w+O(\delta^{w-1}),
\end{align*}
where $\displaystyle c_\ell=\sum_{\substack{i_1+i_2+\cdots +i_\ell= w \\i_1,i_2,\ldots, i_\ell>0}} \lambda_{i_1,i_2,\ldots,i_\ell}$.
When $\ell=w$, the only index for the summand is $i_1=i_2=\cdots=i_\ell=1$ and hence, $\Psi(\vvv)=\delta^j(\delta+1)^{w-j}\ge \delta^w$. Thus, the proof is completed.
\end{proof}
%



In addition, we introduce the notion of maximal-support words.

\begin{definition} Given $V\subseteq \cS$, we say that $\vvv\in V$ is {\bfit maximal-support in $V$} 
(or a {\bfit maximal-support word of $V$}) if
$V$ does not contain a proper ancestor of $\vvv$.
\end{definition}

Recall that $X$ is the smallest set in $\cX_{i-1}$. We have the following technical lemmas.

\begin{lemma} Any maximal-support word in $ X$ is of the form $\vone \bbb$, 
where $\vone$ is the all-ones vector of length $i-1$.
\end{lemma}

\begin{lemma}\label{lem:lem21} A maximal-support word in $ X$ has weight at least $w+1$.
\end{lemma}

\begin{proof}It suffices to consider the case $i-1\le w$. 
We prove by contradiction and assume that $\uuu=\vone \bbb$ is a maximal-support word of weight at most $w$.
and where $\vone$ is the all-ones vector of length $i-1$.

Consider the set $U=\{ \aaa \bbb: \aaa\in\{0,1\}^{i-1} \}$. 
We claim that $ X\setminus U$ is both $(i-1)$-balanced and d-closed.

Since both $U$ and $ X$ are $(i-1)$-balanced,  it follows that $X\setminus U$ is $(i-1)$-balanced.

On the other hand, consider any word $\aaa \bbb$ in $U$ and suppose $\aaa' \bbb'\in X$ 
is an ancestor of $\aaa \bbb$.
In other words, $ \bbb \prec  \bbb'$. Since $\vone \bbb$ is a maximal-support word, we have that $ \bbb'= \bbb$
and so, $\aaa' \bbb'\in U$. Hence,  $ X\setminus U$ is d-closed.

Since all words in $U$ have weight at most $w$, $\Psi(\vvv)\ge 1$ for all $\vvv\in U$.
Therefore, $\Xi_{X\setminus U}(U)=\Psi(U)\ge |U|$, contradicting the removal lemma.
\end{proof}

\begin{lemma}\label{lem:cond1}
Given $w$ and $m\ge 2w$, if  
\begin{equation}\label{eq:cond1}
\delta^w\ge 2^{2w-1}\sum_{j=0}^{w-1} \binom{m-w}{j}
\end{equation}
\noindent then any word $\uuu$ of weight at most $w$ is contained in some maximal-support word of weight at least $2w$.
\end{lemma}

\begin{proof}We prove by contradication, i.e. suppose that all maximal-support words have weight at most $2w-1$.
From Lemma~\ref{lem:lem21}, we have that $\uuu\prec \vvv$ for some word $\vvv$ with weight $w$.
Write $\vvv=\vvv_a\vvv_b$ with $|\vvv_a|=i-1$.

As before, we set $U=\{ \aaa \bbb\in X: \aaa\in\{0,1\}^{i-1}, \vvv_b\prec  \bbb\}$. 
We claim that $ X\setminus U$ is both $(i-1)$-balanced and d-closed.

Since both $U$ and $ X$ are $(i-1)$-balanced,  it follows that $ X\setminus U$ is $(i-1)$-balanced.

On the other hand, consider any word $\aaa \bbb$ in $U$ and suppose $\aaa' \bbb'\in X$ is an ancestor of $\aaa \bbb$.
In other words, $ \bbb \prec  \bbb'$ and so, $\vvv_b \prec  \bbb'$.
Therfore, $\aaa' \bbb'\in U$, and hence,  $ X\setminus U$ is d-closed.

Next, we provide an upper bound for the set $|U|$. Consider $\aaa \bbb\in U$.
Since all maximal-support words of $\vvv$ has weight at most $2w-1$, the weight of $ \bbb$ is at most $2w-i$.
Let the weight of $\vvv_b$ be $w_b$ and so, $w_b\ge w-i+1$. Since $\vvv_b \prec  \bbb$, the number of choices for $ \bbb$ is
\[ \sum_{j=0}^{2w-i-w_b} \binom{m-w_b-i+1}{j}\le \sum_{j=0}^{w-1} \binom{m-w}{j}.\]
Hence, $|U|\le 2^{i-1}\sum_{j=0}^{w-1} \binom{m-w}{j}\le 2^{2w-1}\sum_{j=0}^{w-1} \binom{m-w}{j}$.

On the other hand, since $ X\setminus U$ is d-closed, we have $\Xi_{X\setminus U}(U)=\Psi(U)\ge \Psi(\vvv)\ge \delta^w$.
Then \eqref{eq:cond1} contradicts the removal lemma that states $|U|>\Xi_{X\setminus U}(U)$.
\end{proof}

Given the existence of such maximal-support words, we now make estimates on $\Psi(Y)$ and $\Psi(Z)$.
For convenience, we partition $Y=\bigcup_{\ell=0}^{m-1} Y_\ell$ and $Z=\bigcup_{\ell=1}^m Z_\ell$ 
such that $Y_\ell$ and $Z_\ell$ are words of weight $\ell$ in $Y$ and $Z$, respectively.

The following lemma is an immediate consequence from the definition of $Y$ and $Z$.

\begin{lemma}For $0\le \ell\le m-1$, $|Y_\ell|=|Z_{\ell+1}|$.
The number of words with weight $\ell$ in $Y$ is equal to the number of words
with weight $\ell +1$ in $Z$.
\end{lemma}

The following lemma is a  consequence from Lemma~\ref{lem:cond1}.

\begin{lemma}
If\,\,\eqref{eq:cond1} holds, then
for $1\le \ell\le w$,  we have that $|Y_\ell|>|Y_{\ell-1}|$.
\end{lemma}

\begin{proof}
We form a bipartite graph $\cG=(Y_\ell \cup Y_{\ell-1} , E)$. 
Two vertices $\vvv_1 \in Y_\ell$ and $\vvv_2 \in Y_{\ell-1}$ are connected by an edge if 
$\vvv_2\prec \vvv_1$.
Observe that the degree of a vertex in $Y_\ell$ is at most $\ell$. 
On the other hand, the degree of a vertex $\uuu$ in $Y_{\ell-1}$ is at least $2w - \ell +1 > \ell$ 
since Lemma~\ref{lem:cond1} provides a maximal support word $\vvv$ of weight at least
$2w$ such that $\uuu\prec \vvv$. 
Since the sum of degrees in $Y_\ell$ is equal the sum of degrees in $Y_{\ell-1}$ and 
each vertex in $Y_\ell$ has a smaller degree than a vertex in $Y_{\ell-1}$,
it follows that $|Y_\ell|>|Y_{\ell-1}|$.
\end{proof}

\begin{corollary}
If\,\,\eqref{eq:cond1} holds, then
for $1\le \ell\le w$, we have that $|Y_\ell|>|Z_{\ell}|$.
\end{corollary}

We now prove the main result on asymptotic existence.

\begin{proof}[Proof of Theorem~\ref{theorem-main}]
Applying Lemma~\ref{lem-Psi-equitable}, we have that 
\[
\Psi(Y)  = \sum_{\ell=0}^m \sum_{\vvv\in Y_\ell} \Psi(\vvv)\\
 = \sum_{\ell=0}^w |Y_\ell| \left(c_\ell\delta^w+O(\delta^{w-1})\right)
 \ge \sum_{\ell=1}^w |Y_\ell| \left(c_\ell\delta^w+O(\delta^{w-1})\right). \]
 Similar manipulations yield
 \[ \Psi(Z)  
 = \sum_{\ell=1}^w |Z_\ell| \left(c_\ell\delta^w+O(\delta^{w-1})\right). \]
Since $|Y_\ell|-|Z_\ell|\ge 1$, we estimate the difference $\Psi(Y)-\Psi(Z)$  by 
\begin{equation}\label{eq:condYZ}
\Psi(Y)-\Psi(Z) \ge \sum_{\ell=1}^w c_\ell\delta^w +O(\delta^{w-1}).
\end{equation}

Therefore, for fixed values of $w$, if we choose $m$ sufficiently large such that  \eqref{eq:cond1} holds and  
the right hand side\eqref{eq:condYZ} is nonnegative, we have that $\Psi(Z) \le \Psi(Y)$. In other words, we establish \eqref{eq:target2} and complete our induction argument. Hence, no bad sets exist in the associated compatibility graph 
and a perfect matching or an $(m,n,w)$-domination mapping exists.
\end{proof}

Unfortunately, estimating the values of $c_\ell$ is difficult and 
hence, we are unable to estimate a lower bound of $m$ for which domination mapping exists.
Nevertheless, for the remaining part of the section, we consider the case when $m$ divides $n$ 
and demonstrate that the requirement defined by \eqref{eq:cond1} is mild.
In this case, the domination graph is regular where each vertex in $[m]$ has degree exactly $\delta$.
Then it follows from symmetry that $\Psi(\vvv)$ is dependent only on the weight of $\vvv$.
In other words, for any word $\vvv$ of weight $\ell$, we can write $\Psi(\vvv)$ as $\Psi_\ell$.
Therefore,

\[\Psi(Y)=\sum_{\ell=0}^w |Y_\ell|\Psi_\ell \ge \sum_{\ell=1}^w |Y_\ell|\Psi_\ell
> \sum_{\ell=1}^w |Z_{\ell}|\Psi_\ell 
=\Psi(Z). \]

Hence, we have the following corollary.

\begin{corollary}\label{cor-main}
If\,\,\eqref{eq:cond1} holds,
then $\Psi(Z) < \Psi(Y)$.
\end{corollary}

The next theorem provide certain sufficient numerical conditions for the existence of $(m,\delta m, w)$-domination mappings and imply Theorem~\ref{theorem-main}.

\begin{theorem}
Let $w\ge 3$. Let $\epsilon>0$ and set $N_\epsilon$ such that $m^{\epsilon/(1+\epsilon)}\ge 1+\log m$ for $m\ge N_\epsilon$.
If 
\begin{align}
m & \ge \max\left\{(2w)^{1+\epsilon}, N_\epsilon\right\}, \label{eq:cond2}\\
2^m & \le \sum_{j=0}^w\binom{\delta m}{j}, \label{necessary}
\end{align}
then \eqref{eq:cond1} holds and therefore, an $(m,\delta m,w)$-domination mapping exists.
\end{theorem}

\begin{proof}
Since $m\ge 2w$, 
observe that $\sum_{j=0}^w\binom{\delta m}{j}\le (w+1)\binom{\delta m}{w}=(w+1)(\delta m)^w/w!$\, .
Hence, \eqref{necessary} implies that 
\begin{equation}
\delta^w\ge 2^m \frac{(w+1)!}{wm^w}.\label{eq:5}
\end{equation}
On the other hand,  \eqref{eq:cond2} implies that 
\[
m  = m^{\frac{1}{1+\epsilon}}m^{\frac{\epsilon}{1+\epsilon}}  \ge 2w(1+\log m) \ge 2w\log (2m). 
\]
Hence, $2^m\ge (2m)^{2w}$. Together with \eqref{eq:5}, we have
\[ \delta^w \ge  2^{2w} \frac{w! m^{w}}{(w+1)}\ge 2^{2w-1}\frac{wm^w}{(w-1)!} 
\ge 2^{2w-1} w\binom{m-w}{w-1}
\ge 2^{2w-1}\sum_{j=0}^{w-1} \binom{m-w}{j}.\]
Note that $2w!/(w+1)\ge w/(w-1)!$ for $w\ge 3$. Therefore, \eqref{eq:cond1} holds as desired.
Corollary~\ref{cor-main} then yields \eqref{eq:target2}, which in turn completes our induction argument.
\end{proof}




\appendix


\section{Constructions and Descendant Arrays}

\subsection{Descendant arrays}
\label{appendix-A1}

We will now make use of a different representation of $(m,n,w)$-domination mappings.
For two vectors of the same length $v$ and $u$ we say that
$v$ \emph{covers} $u$ (or $v$ \emph{dominates} $u$; or $u$ is a \emph{descendant} of $v$) and
denote it by $u \prec v$
if $u$ has a \emph{zero} in each entry in which $v$ has a \emph{zero}.
In other words, the value of $u$ in each position is less or equal from the
value of $v$ in the same position.
Let $A$ be an $r \times \ell_1$ binary array and $B$ be an $r \times \ell_2$ binary array, $\ell_2 \geq \ell_1$.
If for each column in $B$ there exists a column in $A$ which dominates it, then
$(A,B)$ is called a pair of $(\ell_1,\ell_2)$-\emph{descendant arrays}.
Let $A$ be a $2^m \times m$ matrix which contains all the binary words of length $m$ (in lexicographic order)
and let $B$ be a $2^m \times n$ matrix which contains distinct binary words
of length~$n$ and weight at most~$w$. This pair of arrays is called
a pair of $(n,m,w)$-descendant arrays if $(A,B)$ is a pair of $(n,m)$-descendant arrays.
By definition we infer that

\begin{lemma}
\label{lem:col_map}
An $(m,n,w)$-domination mapping exists if and only if there exists
a pair of $(m,n,w)$-descendant arrays.
\end{lemma}
\begin{proof}
This is an immediate observation from the definition of the domination graph and defining\linebreak
$\varphi (v_1,v_2,\ldots,v_m) = (u_1,u_2,\ldots,u_n)$, where
$(v_1,v_2,\ldots,v_m)$ and $(u_1,u_2,\ldots,u_n)$ are the $j$th words of the matrices
$A$ and $B$, respectively.
\end{proof}

The description which leads to
Lemma~\ref{lem:col_map} can be used as an alternative way to define the injective
mapping $\varphi (m,n,w)$ and to verify whether such a mapping is a domination mapping.

\subsection{Construction of domination mappings for $w=2$}
\label{appendix-A}


In this appendix we consider $(m,n,2)$-domination mappings for which $n=\nu (m,w)$ if $m$ is odd
and discuss the case of even $m$. We start with odd $m$.
We claim that for odd $m=2\ell+1$ there exists a $(2\ell+1,2^{\ell+1},2)$-domination mapping.
Note, that $\sum_{i=0}^2 \binom{2^{\ell+1}}{i} = 2^{2\ell +1} + 2^\ell +1$ and
$\sum_{i=0}^2 \binom{2^{\ell+1}-1}{i} = 2^{2\ell +1} - 2^\ell +1$, and hence by Lemma~\ref{lem:sum_cond}
we have that a $(2\ell+1,2^{\ell+1},2)$-domination mapping is optimal.
A recursive construction for such a domination mapping will be given.
We start with the $(3,4,2)$-domination mapping from Example~\ref{Example1}.

Assume now that there exists a $(2\ell-1,2^\ell,2)$-domination mapping $\varphi$
with the degree sequence\linebreak $(1,1,2,2,2,4,4,8,8,\ldots,2^{\ell-3},2^{\ell-3},2^{\ell-2},2^{\ell-2})$.
We will construct a $(2\ell+1,2^{\ell+1},2)$-domination mapping $\varphi'$
with the degree sequence $(1,1,2,2,2,4,4,8,8,\ldots,2^{\ell-2},2^{\ell-2},2^{\ell-1},2^{\ell-1})$.

Let $(A(\varphi),B(\varphi))$ be a pair of $(2\ell-1,2^\ell,2)$-descendant arrays related to the
domination mapping~$\varphi$. We will describe now a construction for a pair of $(2\ell+1,2^{\ell+1},2)$-descendant
arrays $(A(\varphi'),B(\varphi'))$ from which $\varphi'$ can be derived.
Let $A_1 A_2$ be a $2^{2\ell-1} \times (2\ell+1)$ matrix, where $A_1$ has two columns
and $A_2$ has $2\ell -1$ columns; $A_1A_2$ represents a quarter of the matrix $A(\varphi')$,
i.e. $A_1$ has one of the for values 00, 01, 10, or 11, and
$A_2$ has $2^{2\ell -1}$ rows with exactly all the words in $\F_2^{2\ell -1}$ in the lexicographic order.
Let $B_1 B_2$ be a $2^{2\ell-1} \times 2^{\ell+1}$ matrix, where $B_1$ has $2^\ell$ columns and
$B_2$ has $2^\ell$ columns; $B_1B_2$ represents a quarter of the matrix $B(\varphi')$.
Furthermore, let $B_1 = B_1^0 B_1^1$, where $B_1^0$ and $B_1^1$ are $2^{2\ell-1} \times 2^{\ell-1}$ matrices.
The $i$th row of $A_1 A_2$ will be mapped by $\varphi'$ to the $i$th row of $B_1 B_2$.
We distinguish now between four cases related to the values of the two columns of $A_1$.

\begin{enumerate}[{[D1]}]
\item If the two columns of $A_1$ are 00, then $B_1$ is the all-zeroes matrix and
$(A_2,B_2)$ is a pair of $(2\ell-1,2^\ell,2)$-descendant arrays with the degree sequence
\linebreak $(1,1,2,2,2,4,4,8,8,\ldots,2^{\ell-3},2^{\ell-3},2^{\ell-2},2^{\ell-2})$.

\item If the two columns of $A_1$ are 11, then $B_2$ is the all-zeroes matrix, and
$B_1$ can be any matrix whose rows are different, $2^{2\ell-1} - 2^{\ell-1}$ rows have weight two, $2^{\ell-1}$ rows have
weight one ($2^{\ell-2}$ of these \emph{ones} in the last $2^{\ell-2}$ columns of $B_1^0$
and the other $2^{\ell-2}$ \emph{ones} in the last $2^{\ell-2}$ columns, which are the
most significant bits, of $B_1^1$).
Simple enumeration yields that there are $2^{2\ell -1}$ such possible different rows as required.

\item If the two columns of $A_1$ are 01 then the matrix $B_2$ is chosen in a way that
each column, except for the first one (least significant bit)
has exactly $2^{\ell-1}$ \emph{ones} and the first column
has $2^{\ell-2}$ \emph{ones}; each row, except for $2^{\ell-2}$ rows of $B_2$, has exactly one \emph{one}
and these $2^{\ell-2}$ rows are all-zeroes rows.
The distribution of the \emph{ones} is done in a way
that the requirements of the related domination graph,
i.e. the degree sequence of $(A_2,B_2)$ as a pair of $(2\ell-1,2^\ell)$-descendant arrays
is\linebreak $(1,1,2,2,2,4,4,8,8,\ldots,2^{\ell-3},2^{\ell-3},2^{\ell-2},2^{\ell-2})$.
For each $2^{\ell-1}$ ones in the same column of $B_2$,
the corresponding $2^{\ell-1}$ rows in $B_1$ have unique ones in different $2^{\ell-1}$ columns
in the last $2^{\ell-1}$ columns of $B_1$. The same is done for the first column of $B_2$
for which the all-zeroes rows are added. The concrete definition will be left as an exercise (not completely trivial).

\item If the two columns of $A_1$ are 10 then $B_2$ is exactly as in the case
where the two columns of $A_1$ are 01, In $B_1$ the first $2^{\ell-1}$ columns are swapped with the
last $2^{\ell-1}$ columns compared to the matrix $B_1$ in the case where the first two columns of $A_1$ are 01,
i.e. $B_1^0$ is swapped with $B_1^1$.
\end{enumerate}

Finally, $A(\varphi')$ consists of the four matrices $A_1 A_2$ of these four cases
and $B(\varphi')$ are the related four matrices $B_1 B_2$.
The construction leads to the following result.

\begin{theorem}
\label{thm:w=2}
If $(A(\varphi),B(\varphi))$ is a pair of $(2\ell-1,2^\ell,2)$-descendant arrays then
$(A(\varphi'),B(\varphi'))$ are two $(2\ell+1,2^{\ell+1},2)$-descendant arrays.
\end{theorem}

Theorem~\ref{thm:w=2} implies the existence of a $(2\ell+1,2^{\ell+1},2)$-domination
mapping for odd $m$. What about an $(m,n,2)$-domination mapping for even $m$.
The following $(m,n)$ pairs were found by computer search to form
optimal $(m,n,2)$-domination mappings: $(4,6)$, $(6,11)$, $(8,23)$, $(10,45)$,
$(12,90)$, $(14,181)$, $(16,362)$, $(18,724)$, $(20,1448)$, $(22,2896)$,
$(24,5793)$, $(26,11585)$, $(28,23170)$, and so on. For $m \geq 6$, these
mappings attains the bound of Lemma~\ref{lem:sum_cond}. In this sequence of optimal
domination mappings one can observe that there is no obvious rule and a hence a recursive
construction for optimal domination mapping won't be an easy task. But, optimal domination mappings exist
and an existence proof for such mappings can be given similarly (but with a simpler proof)
to the existence proof in Section~\ref{sec:proof}.

\section{Proof of Proposition~\ref{prop:reduction}}\label{app:reduction}

We provide a detailed proof of Proposition~\ref{prop:reduction}. To this end, we need the following lemmas.

\begin{lemma}
The set of all $A$-preserving permutations is a subgroup of the set of permutations on $[N]$.
\end{lemma}

\begin{lemma}\label{lem:Apreserving}
Suppose that $\pi$ is $A$-preserving. Then $A\xxx^\pi\le \vone$ if and only if
$A\xxx\le \vone$.
\end{lemma}

\begin{proof}
Since the set of all $A$-preserving permutations form a group, 
it suffices to show one direction.

Suppose that $A\xxx\le \vone$. 
Since $P_{\pi_{\rm row}}A P_\pi=A$ for some permutation $\pi_{\rm row}: [M] \to [M]$,
we have that 
\[\vone \ge A\xxx =P_{\pi_{\rm row}}A P_\pi \xxx= P_{\pi_{\rm row}}A \xxx^\pi.\]
Hence, $A \xxx^\pi\le P_{\pi_{\rm row}}^{-1}\vone=\vone$.
\end{proof}

\begin{proof}[Proof of Proposition~\ref{prop:reduction} ]
Set
\[ \xxx^* =\frac{\sum_{\pi\in G_A}\xxx^\pi}{|G_A|}.\]
By Lemma \ref{lem:Apreserving}, we have that $A\xxx^\pi\le \vone$ for all $\pi\in G_A$.
Therefore, $A\xxx^*\le \vone$ . It is also readily verified that $\sum_{i=1}^N x_i^*=\lambda$

Finally, we show that $\xxx^*$ is $\cO$-regular. For all $k\le N$ and $i,j\in O_k$,
we have that 
\[x^*_i=\frac{\sum_{\pi\in G_A}\xxx^\pi_i}{|G_A|}=\frac{\sum_{\pi\in G_A}\xxx_{\pi(i)}}{|G_A|}.\]
On the other hand, since $i,j\in O_k$, there exists a $\pi^*$ such that $\pi^*(j)=i$. 
Hence,
\[x^*_j=\frac{\sum_{\pi\in G_A}\xxx^\pi_j}{|G_A|}=\frac{\sum_{\pi\in G_A}\xxx_{\pi(j)}}{|G_A|}
=\frac{\sum_{\pi\in G_A}\xxx_{\pi\circ \pi^*(j)}}{|G_A|}==\frac{\sum_{\pi\in G_A}\xxx_{\pi(i)}}{|G_A|}.\]
Therefore, $x^*_i=x^*_j$.
\end{proof}


\newpage


\begin{thebibliography}{10}

\bibitem{Bodi:2009}
{\sc R.\,B\"odi} and {\sc K.\,Herr}, 
Symmetries in linear and integer programs,
\textit{arXiv:0908.3329 preprint}, 
(2009).

\bibitem{CEKV18}
{\sc Y.M.\,Chee, T.\,Etzion, H.M.\,Kiah}, and {\sc A.Vardy},
Cooling codes: Thermal-management coding for high-performance interconnects,
\textit{IEEE Trans.\ Inform.\ Theory},
\textbf{64} (2018), 3062--3085.

\bibitem{CEKVW18}
{\sc Y.M.\,Chee, T.\,Etzion, H.M.\,Kiah, A.Vardy}, and {\sc H.\,Wei},
Low-power cooling codes with effici\-ent encoding and decoding,
\textit{Proc.\ 
International Symp.\ Inform.\ Theory}, 
Vail, CO, (2018), 1655--1659.

\bibitem{ES16}
{\sc T.\,Etzion} and {\sc L.\,Storme},
Galois geometries and coding theory,
\textit{Designs, Codes, Crypto.},
\textbf{78} (2016), 311--350.

\bibitem{Hall35}
{\sc P.\,Hall},
On representatives of subsets,
\textit{J. London Math.\ Society}, 
\textbf{10} (1935), 26--30.

\bibitem{Heden}
{\sc O.\,Heden},
A survey of perfect codes,
\textit{Adv.\ Math.\ Commun.}
\textbf{2(2)} (2008), 223--247.

\bibitem{Konig}
{\sc D.\,K\"onig}, 
Gr\'afok \'es m\'atrixok,
\textit{Matematikai \'es Fizikai Lapok}, 
\textbf{38} (1931), 116--119.

\bibitem{MWS}
{\sc F.\kern1ptJ.\,MacWilliams} and {\sc N.\kern1ptJ.A.\,Sloane},
\emph{The Theory of Error-Correcting Codes}.
Amsterdam:~North Holland Publishing Company, May 1978.

\bibitem{Margot:2003}
{\sc F.\, Margot}, 
Exploiting orbits in symmetric ILP, 
\textit{Mathematical Programming}, 
\textbf{98(1)} (2003), 3--21.

\bibitem{sage}
{\sc {S}ageMath, {S}age {M}athematics {S}oftware {S}ystem ({V}ersion 7.6)}, 
Sage Developers,~2017, {\tt http://www.sagemath.org}.

\bibitem{Solovieva}
{\sc F.\kern1ptI.\,Solov'eva},
On perfect binary codes,
\textit{Discr.\ Applied Math.}
\textbf{152(9)} (2008), 1488--1498.

\bibitem{NS17}
{\sc E.\,N\u{a}stase} and {\sc P.\,Sissokho},
The maximum size of a partial spread in a finite projective space,
\textit{J.\ Combin.\ Theory Ser.\:A},
\textbf{152} (2017), 353--362.

\end{thebibliography}
\end{document}